\documentclass{amsart}
\usepackage{amsmath}
\usepackage{amsthm}
\usepackage{titlesec}
\usepackage{mathrsfs}
\usepackage{amssymb}
\usepackage{mathtools}
\usepackage{color}
\usepackage{algorithm}
\usepackage{algorithmic}
\usepackage{cite}
\usepackage{latexsym}
\usepackage{booktabs}
\usepackage{bm}
\usepackage{graphicx}
\usepackage{textcomp}
\usepackage{subfigure}
\newcommand{\Rni}{\mathbb{R}^{n_{i}}}

\newcommand{\nfR}{\mathbb{R}}
\newcommand{\xpi}{x_i}

\newcommand{\xpij}{x_{i,j}}
\newcommand{\xpik}{x_i^{(k)}}
\newcommand{\ypi}{y_i}
\newcommand{\xmi}{x_{-i}}
\newcommand{\allx}{\mathbf{x}}
\newcommand{\fpi}{f_{i}}
\newcommand{\Xpi}{X_{i}}
\newcommand{\lXpi}{X_{i}(x_{-i})}
\newcommand{\lip}{\left<}
\newcommand{\rip}{\right>}

\newcommand{\qmod}{\mbox{Qmod}}

\titleformat{\section}[hang]{\large\bfseries}{\thesection.}{1em}{}
\titleformat{\subsection}{\normalsize\bfseries}{\thesubsection}{1em}{}

\DeclareMathOperator{\Rank}{rank}
\DeclareMathOperator{\dom}{dom}

\theoremstyle{remark}
\newtheorem*{remark}{Remark}

\newcommand{\re}{\mathbb{R}}

\newcommand{\N}{\mathbb{N}}
\def\af{\alpha}

\newcommand{\st}{\mbox{s.t.}}
\newcommand{\reff}[1]{(\ref{#1})}

\newcommand{\eps}{\epsilon}

\newcommand{\bdes}{\begin{description}}
\newcommand{\edes}{\end{description}}

\newcommand{\bal}{\begin{align}}
\newcommand{\eal}{\end{align}}

\newcommand{\bnum}{\begin{enumerate}}
\newcommand{\enum}{\end{enumerate}}

\newcommand{\bit}{\begin{itemize}}
\newcommand{\eit}{\end{itemize}}

\newcommand{\bea}{\begin{eqnarray}}
\newcommand{\eea}{\end{eqnarray}}
\newcommand{\be}{\begin{equation}}
\newcommand{\ee}{\end{equation}}

\newcommand{\baray}{\begin{array}}
\newcommand{\earay}{\end{array}}

\newcommand{\bsry}{\begin{subarray}}
\newcommand{\esry}{\end{subarray}}

\newcommand{\bca}{\begin{cases}}
\newcommand{\eca}{\end{cases}}

\newcommand{\bcen}{\begin{center}}
\newcommand{\ecen}{\end{center}}

\newcommand{\bbm}{\begin{bmatrix}}
\newcommand{\ebm}{\end{bmatrix}}

\newcommand{\btab}{\begin{tabular}}
\newcommand{\etab}{\end{tabular}}

\theoremstyle{definition}
\newtheorem{definition}{Definition}[section]

\theoremstyle{plain}

\newtheorem{lemma}[definition]{Lemma}
\newtheorem{theorem}[definition]{Theorem}

\newtheorem{example}[definition]{Example}
\newtheorem{alg}[definition]{Algorithm}

\numberwithin{equation}{section}

\begin{document}

\title[The Gauss-Seidel Method for GNEPPs]
{The Gauss-Seidel Method for Generalized Nash Equilibrium Problems of Polynomials}

\author[Jiawang Nie]{Jiawang~Nie}
\address{Jiawang Nie, Xindong Tang, Department of Mathematics,
University of California San Diego,
9500 Gilman Drive, La Jolla, CA, USA, 92093.}
\email{njw@math.ucsd.edu, xit039@ucsd.edu}

\author[Xindong Tang]{Xindong~Tang}

\author[Lingling Xu]{Lingling~Xu}
\address{School of Mathematical Sciences, Nanjing Normal University,
Jiangsu Key Laboratory for NSLSCS, Nanjing Normal University, Nanjing,
Jiangsu, China, 210023.}
\email{xulingling@njnu.edu.cn}


\date{}


\begin{abstract}
This paper concerns the generalized Nash equilibrium problem of polynomials (GNEPP). 
We apply the Gauss-Seidel method and Lasserre type Moment-SOS relaxations 
to solve GNEPPs. The convergence of the Gauss-Seidel method is known for some special GNEPPs,
such as generalized potential games (GPGs).
We give a sufficient condition for GPGs and propose a numerical certificate, 
based on Putinar's Positivstellensatz.
Numerical examples for both convex and nonconvex GNEPPs
are given for demonstrating the efficiency of the proposed method. 
\end{abstract}

\maketitle

\section{Introduction}
\label{intro}
The generalized Nash equilibrium problem (GNEP) is a kind of games
such that the feasible set of each player's strategy depends on other players' strategies.
Let $N$ be the number of players. Suppose the $i$th player's strategy is the variable
$x_{i}\in\mathbb{R}^{n_{i}}$
(the $n_i$-dimensional Euclidean space over the real field $\re$).
The vector of all players' strategy is
\[
\allx \, := \, (x_1, \ldots, x_N).
\]
The total dimension of all players' strategies is $n := n_1 + \cdots + n_N.$
For convenience, we use $x_{-i}$ to denote the subvector of all players' strategies
except the $i$th one, i.e.,
\[
x_{-i} \, := \, (x_1, \ldots, x_{i-1}, x_{i+1}, \ldots, x_N).
\]
When the $i$th player's strategy $x_i$ is considered,
we use $(x_{i},x_{-i})$ to represent $\allx$.
% and write that $\allx = (x_i, x_{-i})$.
When $(y,x_{-i})$ is written, it means that the $i$th player's strategy is
$y\in\Rni$ while the vector of all other players' strategies is fixed to be $\xmi$. 
This paper considers GNEPs whose objective and constraining functions
are given by polynomials.

\begin{definition}
A generalized Nash equilibrium problem of polynomials (GNEPP) is to
find $\allx\in\nfR^{n}$ such that each $\xpi$ is an optimizer of the
$i$th player's optimization problem
\be \label{GNEP}
 \left\{ \begin{array}{cl}
 \min\limits_{\xpi\in \re^{n_i}}  &  \fpi(\xpi,\xmi) \\
 \st & g_{i,j}(\xpi, \xmi) \geq 0 \, (j=1,\ldots, s_i),
\end{array} \right.
\ee
where all $f_i(\xpi,\xmi)$ and $g_{i,j}(\xpi,\xmi)$
are polynomial functions in $\allx$.
A solution $\allx$ satisfying the above is
called a generalized Nash equilibrium (GNE).
\end{definition}

Let $g_i=(g_{i,1},\ldots,g_{i,s_i}): \mathbb{R}^n \rightarrow \mathbb{R}^{s_i}$
be the vector-valued function. The inequality $g_i(\xpi, \xmi) \ge 0$
is defined componentwisely. Then \reff{GNEP} can be rewritten as
\be
\label{subprob}
%\tag{$\mathcal{SP}$-i}
\left\{ \begin{array}{cl}
\min\limits_{\xpi\in \re^{n_i}}  &  \fpi(\xpi,\xmi) \\
\st & g_i(\xpi, \xmi) \geq 0.\\
\end{array} \right.
\ee
For given $\xmi$, the feasible strategy set for the $i$th player is
\be \label{Xi(x-i)}
X_i(\xmi) \, := \,
\{x_i \in \re^{n_i}: g_i(\xpi, \xmi) \geq 0  \} .
\ee
The entire strategy vector $\allx$ is said to be a feasible point if each subvector
$\xpi$ of $\allx$ is feasible for (\ref{subprob}).
For instance, the following GNEP with two players
\be
\label{convnotsol}
\begin{array}{ccccc}
\min\limits_{x_1 \in \re^1}& x_1 & \vline &
\min\limits_{x_2 \in \re^1} & (x_2)^2-(x_1-1)x_2 \\
\st & x_2(x_1-x_2-1)\ge 0,  & \vline & \st &(x_1)^2+(x_2)^2 \le 3,  \\
    & x_1\ge 0, & \vline & & x_2\ge 0
\end{array}
\ee
is a GNEPP. The dimensions $n_1 = n_2 = 1$ and
\[
\begin{array}{l}
X_1(x_{-1})=\{x_1\in\re: x_2(x_1-x_2-1)\ge 0,x_1\ge0\}, \\
X_2(x_{-2})=\{x_2\in\re: (x_1)^2+(x_2)^2 \le 3,x_2\ge0 \}.
\end{array}
\]
For the first player, when $x_2>0$, its feasible set is $x_1\ge x_2+1$,
and its best strategy is $x_1=x_2+1$.
When $x_2 = 0$, the first player's  best strategy is $x_1=0$.
For any fixed $x_1$ with $x_1^2 \leq 3$,
the second player's problem is feasible and its best strategy is
$\max ( (x_1-1)/2, 0 ).$
One can verify that  $(0,0)$ is a GNE for this GNEPP.

Generalized Nash equilibrium problems have broad applications, for instance,
in the environmental pollution control \cite{Facchinei2007,breton2006game}.
Let $N$ be the number of countries involved in the pollution control and
$x_{i,0}$ denote the (gross) emissions from the $i$th country.
Assume that the by-product gross emissions are proportional to the industrial output.
The revenue of the $i$th country depends on $ x_{i,0}$.
Typically, the revenue is $x_{i,0}(b_i-1/2x_{i,0})$ with a given parameter $b_i$.
The variable $x_{i,j} $ represents the investment from country $i$ to country $j$.
Let $x_i := (x_{i,0},\dots,x_{i,N})$.
For an investor, the benefit of the investment lies in the
\emph{emissions reduction units} $\gamma_{i,j}x_{i,j}$ with given parameters
$\gamma_{i,j}  (i, j= 1, \cdots, N)$.
The net emission in country $i$ is $x_{i,0}-\sum_{j=1}^N\gamma_{j,i}x_{j,i}$,
which is always nonnegative.
The accounted-for-emissions for the $i$th country is
$x_{i,0}-\sum_{j=1}^N\gamma_{i,j}x_{i,j}$.
It must be kept below or equal a certain prescribed level $E_i$
under the environmental control.
The pollution in a country may affect other countries.
The pollution damage for the $i$th country is
 \[
 p_i \, := \,
 x_{i,0}-\sum\limits_{j=1}^N\gamma_{j,i}x_{j,i}+2\prod_{k=1}^N (x_{k,0}-\sum\limits_{j=1}^N\gamma_{j,k}x_{j,k}).
 \]
For given parameters $b_i,\gamma_{i,j},E_i$,
the $i$th country's optimization problem is
\begin{equation}
\left\{ \begin{array}{ll}
 \min\limits_{ x_i } &-x_{i,0}(b_i-\frac{1}{2} x_{i,0})
            +\sum\limits_{j=1}^{N}x_{i,j}+ p_i \\
    \st &x_{i,0}\dots x_{i,j}\ge0,\\
    &x_{i,0}-\sum_{j=1}^N\gamma_{i,j}x_{i,j}\le E_i,\\
    &x_{k,0}-\sum_{j=1}^N\gamma_{j,k}x_{j,k}\ge0\ (k=1,\dots,N).
  \end{array} \right.
\end{equation}
All countries expect to maximize their revenues subtracting
investments and pollution damages.
Another application of GNEPP is the model for Internet switching
(see Example~\ref{internet}).
More applications for GNEPPs can be found in
\cite{Contreras2004,Couzoudis2013,Paccagnan2016,Yin2009,Anselmi2017}.

\subsection{GNEPs and some existing work}

The GNEP is an extension of the Nash equilibrium problem (NEP)
\cite{10.2307/1969529,nobel}. For NEPs,
the feasible set of each player's strategy is independent of other players.
The GNEP originated from economics and was studied in
\cite{debreu1952social,arrow1954existence,basar1999dynamic,McKenzie1959,rosen1965}.
Robinson \cite{robinson1993shadow1, robinson1993shadow2}
established the shadow prices
for measuring the effectiveness in an optimization-based combat model. 
Scotti \cite{scotti1995structural} introduced GNEPs into the study of structural design.
Recently, GNEPs have been widely used in many different areas outside economics,
such as  transportation, telecommunications, pollution control.
We refer to
\cite{ardagna2017generalized,breton2006game, oggioni2012generalized,
sun2007equilibrium,yue2014game,zhou2005generalized}
for related work.

The following is a classical result about existence of solutions for GNEPs \cite{debreu1952social,Facchinei2007}.
We refer to \cite{rockafellar2009variational}
for the notion of outer and inner semicontinuity and quasi-convexity.
\begin{theorem}
\cite{debreu1952social,Facchinei2007}  \label{exstc}
Suppose the GNEP of \reff{GNEP} satisfies:
\begin{enumerate}
\item There exist $N$ nonempty, convex and compact sets $K_i\subseteq \mathbb{R}^{n_i}$
such that for every $(\xpi,\xmi)\in\mathbb{R}^n$ with $x_i\in K_i$ and for every $i$,
the set $X_i(x_{-i})$ is nonempty, closed and convex, $X_i(x_{-i})\subseteq K_i$,
and $X_i(\,\cdot\,)$, as a point-to-set map, is both outer and inner semicontinuous.

\item 
For every given $x_{-i}$, the function $f_i(\,\cdot\,,x_{-i})$
is quasi-convex on $X_i(x_{-i})$.
\end{enumerate}
Then, a generalized Nash equilibrium exists.
\end{theorem}

There are no special existence results for solutions of GNEPPs,
to the best of the authors' knowledge.
There exists some work for solving GNEPs.
Under some convexity assumptions,
the GNEP is equivalent to a quasi-variational inequality problem(QVIP)
\cite{ facchinei2007generalized,harker1991generalized,
nabetani2011parametrized,Pang2005,Aussel2017}.
The Karush-Kuhn-Tucker (KKT) optimality conditions
for each player's optimization problem
%can be put together, which forms nonlinear equations.
can be used together with the semismooth Newton-type method \cite{dreves2014new,dreves2011solution,facchinei2009generalized}.
A GNEP can be transformed to a NEP with the usage of penalty functions
\cite{FacKan10,fukushima2011, kanzow2016,Facchinei2011partial}.
Gap functions are frequently used for solving GNEPs \cite{vonHeusinger2009-2}.
A relaxation method for jointly convex GNEPs,
based on inexact line search and Nikaido-Isoda functions,
is given in \cite{vonHeusinger2009}.
A study on GNEPs with linear coupling constraints and
mixed-integer variables is in \cite{Sagratella2019}.
Facchinei et al.~\cite{Facchinei2011} proposed the Gauss-Seidel method
for solving GNEPs.
Its main idea is to solve each player's optimization problem alternatively.
We also refer to \cite{Sagratella2016computing,Sagratella2017computing}
for studies on the Gauss-Seidel method for solving GNEPs
with discrete and mixed integer variables.
Convergence of the Gauss-Seidel method can be shown
for some special GNEPs, such as generalized potential games (GPGs).
We refer to \cite{Facchinei2011,Monderer1996,Sagratella2017algorithms}
for studies on potential games and GPGs.
Most of the existing methods assume that
each individual player's optimization problem is convex.
For more work about GNEPs, we refer to the surveys
\cite{Facchinei2007, fischer2014generalized}.

\subsection{Contributions}

In this paper, we use the Gauss-Seidel method introduced in \cite{Facchinei2011}
for solving GNEPPs.
This method requires to get global minimizers
for the occurring optimization problems in each loop.
Although the Gauss-Seidel method is not theoretically 
guaranteed to converge for all GNEPPs, 
it converges for many problems in the computational practice.
For some special GNEPs, the convergence can be guranteed for the Gauss-Seidel method.
Generalized potential games (GPGs) are such GNEPs \cite{Facchinei2011}.
We would like to remark that the Gauss-Seidel method
can be applied to a GNEPP even if it is not a GPG,
while the convergence is not guaranteed. In practice,
the method works well for many GNEPPs that are not GPGs.
In \cite{Facchinei2011}, it was shown that if a GNEP is in some special forms,
then it is a GPG (see section~\ref{sc:GPG}).
In section \ref{sc:GPG}, we give the first numerical method
for verifying GPGs. Our major results are:
\begin{itemize}

\item We use the Lasserre type Moment-SOS relaxations \cite{lasserre2015introduction}
to find global minimizers of the occurring polynomial optimization
problems in each loop of the Gauss-Seidel method.
These relaxations can solve the polynomial optimization problems globally,
even if they are nonconvex.

\item As demonstrated in section \ref{nr},
the Gauss-Seidel method works well in practice.
Moment-SOS relaxations can be used to verify if a computed solution
is a GNE or not.
There are no other numerical methods for solving GNEPPs efficiently,
especially for nonconvex ones,
to the best of the authors' knowledge.

\item We give a sufficient condition for checking if a given GNEPP is a GPG or not.
Based on it, a numerical certificate is given for checking GPGs.
This is the first numerical method that can do this,
to the best of the authors' knowledge.
\end{itemize}

The paper is organized as follows.
Some preliminaries about polynomial optimization
are given in Section~\ref{sc:pre}.
Section~\ref{sc:GSAlg} gives the Gauss-Seidel method for solving GNEPPs
and studies its properties, and the algorithm of
solving the occurring polynomial optimization problems globally in the Gauss-Seidel method.
Section~\ref{sc:GPG} focuses on generalized potential games.
Numerical experiments are given in Section~\ref{nr}.

\section{Preliminaries}
\label{sc:pre}

The symbol $\mathbb N$ stands
for the set of nonnegative integers, and $\mathbb R$ for the real field.
The norm $ \|\cdot\|$ is the standard Euclidean norm of a vector.
For a real number $t$, $\lceil t \rceil$ (resp., $\lfloor t \rfloor$)
denotes the smallest integer not smaller than $t$
(resp., the biggest integer not bigger than $t$).
%Let $N$ be the number of players.

The variable $x_i \in \re^{n_i}$
is the strategy of the $i$th player,
and $x_{i,j}$ denotes the $j$th component of $x_i$,
for $j = 1, \ldots, n_i$.
%
%The total dimension of all strategy variables is
%$n = n_1 + \cdots + n_N$.
%
In the Gauss-Seidel method, we use
$\xpik$ to denote the value of $x_i$ in the $k$th loop.
Similarly, $\xpij^{(k)}$ denotes $\xpij$ in the $k$th iteration.
In each loop of the Gauss-Seidel method,
we need to solve a polynomial optimization problem about
the $i$th player's strategy vector $x_i := (x_{i,1}, \ldots, x_{i,n_i})$.
Let $\nfR[x_i]$ denote the ring of real polynomials in $x_i$,
and for a degree $d$,
$\nfR[x_i]_d$ denotes the space of all polynomials in $x_i$
whose degrees are at most $d$.
For $x_i := (x_{i,1}, \ldots, x_{i,n_i})$
and $\af := (\af_1, \ldots, \af_{n_i}) \in \N^{n_i}$, denote
\[
x_i^\alpha := x_{i,1}^{\alpha_1} \cdots x_{i,n_i}^{\alpha_{n_i}}, \quad
|\alpha|:=\alpha_1+\cdots+\alpha_{n_i}.
\]
For an integer $d >0$, denote the set
\[
{\mathbb{N}}_d^{n_i} \, := \,
\{\alpha\in {\mathbb{N}}^{n_i}: \, \ |\alpha| \le d \}.
\]
We use $[x_i]_d$ to denote the vector of all monomials in $x_i$
and whose degree is at most $d$, ordered in the graded alphabetical ordering.
For example, if $x_i =(x_{i,1}, x_{i,2})$, then
\[
[x_i]_3 = (1,  x_{i,1}, x_{i,2}, x_{i,1}^2, x_{i,1}x_{i,2},
x_{i,2}^2, x_{i,1}^3, x_{i,1}^2x_{i,2}, x_{i,1}x_{i,2}^2, x_{i,2}^3).
\]
A polynomial $\sigma \in \re[x_i]$ is said to be a sum of squares (SOS)
if $\sigma=s_1^2+s_2^2+\dots+s_k^2$ for some polynomials $s_1,\dots,s_k\in\nfR[x_i]$.
The set of all SOS polynomials in $x_i$ is denoted as $\Sigma[x_i]$.
For a degree $d$, we denote the truncation
\[
\Sigma[x_i]_d \, := \, \Sigma[x_i] \cap \nfR[x_i]_d.
\]
For a tuple $g=(g_1,\dots,g_t)$ of polynomials in $x_i$,
its quadratic module is the set
\[
\qmod(g) =  \Sigma[x_i] +  g_1 \cdot \Sigma[x_i] + \cdots + g_t \cdot  \Sigma[x_i].
\]
The truncation of $\qmod(g)$ with degree $2d$ is the set
\[
\qmod(g)_{2d} =\Sigma[x_i]_{2d}+g_1\cdot\Sigma[x_i]_{2d-\deg(g_1)}
+\cdots+g_t\cdot\Sigma[x_i]_{2d-\deg(g_t)}.
\]
The tuple $g$ defines the basic closed semi-algebraic set
\begin{equation}
  \label{polyrep}
  \mathcal{S}(g)= \{x_i \in \nfR^{n_i}: g(x_i) \ge  0  \}.
\end{equation}
The quadratic module $\qmod(g)$ is said to be {\it archimedean}
if there exists $p\in\qmod(g)$ such that the set $\mathcal{S}(p)$ is compact.
If $\qmod(g)$ is archimedean, then $\mathcal{S}(g)$ must be compact.
Conversely, if $\mathcal{S}(g)$ is compact, say,
$\mathcal{S}(g)$ is contained in the ball $\mathcal{S}(R -\|x_i\|^2)$,
then $\qmod(g,R -\|x_i\|^2)$ is archimedean
and $\mathcal{S}(g) = \mathcal{S}(g, R -\|x_i\|^2)$.
For a polynomial $f \in \nfR[\allx]$,
if $f\in \qmod(g)$, then it is clear that $f\ge0$ on $\mathcal{S}(g)$.
The reverse is not necessarily true.
However, when $\qmod(g)$ is archimedean,
if $f > 0$ on $\mathcal{S}(g)$, then $f \in \qmod(g)$.
This conclusion is referred to as Putinar's Positivstellensatz
\cite{putinar1993positive}.
Interestingly, if $f \ge 0$ on $\mathcal{S}(g)$, we still have $f\in \qmod(g)$,
under some optimality conditions \cite{nie2014optimality}.

\section{The Gauss-Seidel method for GNEPPs }
\label{sc:GSAlg}

The Gauss-Seidel method was introduced in \cite{Facchinei2011}
for solving GNEPs. The following is the general framework of the Gauss-Seidel method.

\begin{alg}
%\caption{Gauss-Seidel Method}
\rm    \label{gs}
For the GNEP of \reff{GNEP}, do the following:
\begin{algorithmic}

\STATE Step~1. Choose a feasible starting point
 $\allx^{(0)}=(x_{1}^{(0)},\dots,x_{N}^{(0)})$, a positive regularization parameter
  $\tau^{(0)}$ and let $k:=0$.

\STATE Step~2.  If $\allx^{(k)}$ satisfies a suitable termination criterion, stop.

\STATE Step~3.  For  $i=1,\dots,N,$
compute a global minimizer $x_{i}^{(k+1)}$ of the optimization
\begin{equation}
\left\{ \begin{array}{ll}
    \min\limits_{x_{i} \in \re^{n_i}}& f_{i}(x_{1}^{(k+1)},\dots,x_{i-1}^{(k+1)},x_{i},x_{i+1}^{(k)},\dots,x_{N}^{(k)})
    +\tau^{(k)}\Vert x_{i}-x_{i}^{(k)}\Vert^2\\
\st & g_{i}(x_{1}^{(k+1)}, \dots, x_{i-1}^{(k+1)}, x_{i}, x_{i+1}^{(k)},\dots,x_{N}^{(k)})
        \ge0.
  \end{array} \right.
  \label{subopt}
\end{equation}
%\ENDFOR

% \STATE Set $\tau^{(k+1)}=\max\{\min[\tau^{(k)},\max_{i=1,\dots,N}(\Vert x^{i+1,k}-x^{i,k}\Vert)]\}$.
\STATE Step~4.  Choose a new regularization parameter
%$0\le \tau^{(k+1)}\le\tau^{(k)}.$
$\tau^{(k+1)} \in [0, \tau^{(k)}].$

\STATE Step~5. Let $\allx^{(k+1)}:=(x_{1}^{(k+1)},\dots,x_{N}^{(k+1)})$,
$k := k+1$, and go to Step~2.

\end{algorithmic}
\end{alg}

In practice, Algorithm~\ref{gs} performs well for solving GNEPPs.
It can compute equilibria for many problems.
This is demonstrated in numerical experiments in Section~\ref{nr}.
The GNEPPs are very hard to be solved by other existing methods,
to the best of the authors' knowledge.
On the other hand, Algorithm~\ref{gs} is not theoretically guaranteed
to converge for all GNEPPs.
Its convergence can be shown for some special GNEPs, such as GPGs.
In the following, we show how to implement Algorithm~\ref{gs}
when the defining functions are polynomials.
After that, we review some properties of Algorithm~\ref{gs}.

\subsection{Moment-SOS relaxations for polynomial optimization}

We discuss how to implement Algorithm~\ref{gs} when
all the objective and constraining functions are given by polynomials.
In its Step~3, the sub-optimization \reff{subopt}
is a polynomial optimization problem whose variable is $x_i \in \re^{n_i}$.
In the following, we give a brief review for
using the Lasserre type Moment-SOS hierarchy to solve (\ref{subopt}).
We refer to \cite{Las01,lasserre2015introduction,nie2013certifying,nie2014optimality}
for related work about polynomial optimization.

For an even degree $2d>0$,
let $\re^{ \N_{2d}^{n_i} }$ denote the space of all real vectors
that are labeled by $\af \in \N_{2d}^{n_i}$. Each
$y \in \re^{ \N_{2d}^{n_i} }$ is labeled as
\[
y \, = \, (y_\af)_{ \af \in \N_{2d}^{n_i} }.
\]
Such $y$ is called a
{\it truncated multi-sequence} (tms) of degree $2d$.
For a polynomial $f = \sum_{ \af \in \N^n_{2d} } f_\af x_i^\af \in \re[x]_{2d}$,
define the operation
\be \label{<f,y>}
\langle f, y \rangle = \sum_{ \af \in \N^n_{2d} } f_\af y_\af.
\ee
The operation $\langle f, y \rangle$ is linear in $y$ for fixed $f$
and it is linear in $f$ for fixed $y$.
For a polynomial $q \in \re[x_i]_{2d}$ and the integer
$t = d- \lceil \deg(q)/2 \rceil$, the outer product
$q \cdot [x_i]_t[x_i]_t^T$
is a symmetric matrix of length $\binom{n+t}{t}$.
It can be expanded as
\[
q \cdot [x_i]_t [x_i]_t^T \, = \, \sum_{ \af \in \N_{2d}^n }
x_i^\af  Q_\af,
\]
for some symmetric matrices $Q_\af$. We denote
\be \label{df:Lf[y]}
L_{q}^{(d)}[y] \, := \, \sum_{ \af \in \N_{2d}^n }
y_\af  Q_\af.
\ee
It is called the $d$th {\it localizing matrix} of $q$ and generated by $y$.
For given $q$, $L_{q}^{(d)}[y]$ is linear in $y$.
Clearly, if $q(u) \geq 0$ and $y = [u]_{2d}$, then
\[ L_{q}^{(d)}[y] = q(u) [u]_t[u]_t^T \succeq 0. \]
For instance, if $n_i=d=2$
and $q= 1 - x_{i,1}-x_{i,1}x_{i,2}$, then
\[
L_q^{(2)}[y]=\left [\begin{matrix}
y_{00}-y_{10}-y_{11} &  y_{10}-y_{20}-y_{21} &  y_{01}-y_{11}-y_{12} \\
y_{10}-y_{20}-y_{21} &  y_{20}-y_{30}-y_{31} &  y_{11}-y_{21}-y_{22} \\
y_{01}-y_{11}-y_{12} &  y_{11}-y_{21}-y_{22} &  y_{02}-y_{12}-y_{13} \\
\end{matrix}\right ].
\]
When $q=1$ is the constant one polynomial,
the localizing matrix $L_{1}^{(d)}[y]$
reduces to a moment matrix, which we denote as
\[
M_d[y] \, := \, L_{1}^{(d)}[y].
\]
When $\mathbf{q}:=(q_1,\dots,q_s)$ is a tuple of polynomials,
we then define that
\[
L_{\mathbf{q}}^{(d)}[y]:=\left [\begin{matrix}
L_{q_1}^{(d)}[y] &                   &        &\\
                 &  L_{q_2}^{(d)}[y] &        &\\
                 &                   &  \ddots& \\
                 &                   &        &L_{q_s}^{(d)}[y]\\
\end{matrix}\right ].
\]

In the following, we discuss how to solve (\ref{subopt}).
For convenience, denote
\be \label{simsubopt}
\left\{ \baray{l}
f_i^{(k)} :=f_{i}(x_{1}^{(k+1)},\dots,x_{i-1}^{(k+1)},x_{i},x_{i+1}^{(k)},\dots,x_{N}^{(k)})
+\tau^{(k)}\Vert x_{i}-x_{i}^{(k)}\Vert^2,\\
g_i^{(k)} :=g_{i}(x_{1}^{(k+1)}, \dots, x_{i-1}^{(k+1)}, x_{i}, x_{i+1}^{(k)},\dots,x_{N}^{(k)}).
\earay \right.
\ee
They are polynomials in $x_i$.
One can rewrite (\ref{subopt}) equivalently as
\be  \label{pop:fkgk}
\left\{
\baray{rcl}
\vartheta_{\min}= & \min\limits_{x_i \in \re^{n_i} } & f_i^{(k)}(x_i) \\
& \st & g_i^{(k)}(x_i)\ge 0.
\earay
\right.
\ee
Denote the degree
\[
d_0:=\max\{\lceil\deg(f_i^{(k)})/2\rceil,\lceil\deg(g_i^{(k)})/2\rceil\}.
\]
For $d=d_0,d_0+1,\ldots$, the $d$th moment relaxation for \reff{pop:fkgk} is
\be
\label{d-mom}
\left\{
\baray{rl}
\vartheta_d  \, :=  \,\min\limits_{ y } & \lip f_i^{(k)},y \rip\\
 \st & M_d[y] \succeq 0, \, L_{g_i^{(k)}}^{(d)} \succeq 0, \\
  & y_0=1, y\in\mathbb{R}^{\mathbb{N}^{n_i}_{2d}}.
\earay
\right.
\ee
Its dual optimization problem is the SOS relaxation
\be
\label{d-sos}
\left\{
\baray{ll}
\max & \gamma\\
\st & f_i^{(k)} -\gamma \in  \qmod(g_i^{(k)})_{2d}. \\
\earay
\right.
\ee
We refer to Section~\ref{sc:pre} for the truncated quadratic module
$\qmod(g_i^{(k)})_{2d}$.
By solving the relaxations \reff{d-mom}-\reff{d-sos}
for $d=d_0, d_0+1, \ldots$, we get the Moment-SOS hierarchy
for solving \reff{pop:fkgk}.
The following is the algorithm.

\begin{alg} \label{lh} \rm
(The Moment-SOS hierarchy for solving (\ref{pop:fkgk})).
Let $f_i^{(k)}, g_i^{(k)}$ be as in \reff{simsubopt}.
Start with $d := d_0$.
\begin{algorithmic}

\STATE Step~1. Solve the semidefinite relaxation (\ref{d-mom}).
If (\ref{d-mom}) is infeasible, then \reff{pop:fkgk}
has no feasible points and stop;
otherwise, solve it for a minimizer $y^*$ and let $t:=d_1$,
where $d_1 := \lceil \deg(g_i^{(k)})/2 \rceil$.

\STATE Step~2. If $y^*$ satisfies the rank condition
\be \label{flatrank}
\Rank{M_t[y^*]}=\Rank{M_{t-d_1}[y^*]} ,
\ee
then extract $r :=\Rank{M_t(y^*)}$ minimizers for (\ref{pop:fkgk}) and stop.

\STATE Step~3. If \reff{flatrank} fails to hold and $t<d$,
let $t := t+1$ and then go to Step~2;
otherwise, let $d := d+1$ and go to Step~1.
\end{algorithmic}

\end{alg}

The rank condition \reff{flatrank} is called {\it flat truncation}
in the literature \cite{nie2013certifying}.
It is a sufficient (and almost necessary) condition for checking convergence
of the Moment-SOS hierarchy.
Flat truncation is useful for solving truncated moment problems
and linear optimization with moment constraints
\cite{HelNie12,FiaNie12,linmomopt}.
Indeed, the Moment-SOS hierarchy has finite convergence
if and only if the flat truncation is satisfied for some relaxation order,
under some generic conditions~\cite{nie2013certifying}.
When \reff{flatrank} holds, the method in \cite{HenLas05}
can be used to extract $r$ minimizers for \reff{pop:fkgk}.
The method is implemented in the software {\tt GloptPoly~3}~\cite{GloPol3}.
We refer to \cite{HenLas05}, \cite{nie2013certifying} and
\cite[Chapter 6]{lasserre2015introduction} for more details.

The convergence properties of Algorithm~\ref{lh} are as follows.
By solving the hierarchy of relaxations \reff{d-mom}-\reff{d-sos},
we can get a monotonically increasing
sequence of lower bounds $\{\vartheta_d\}_{d=d_0}^{\infty}$
for the minimum value $\vartheta_{\min}$, i.e.,
\[
\vartheta_{d_0} \le \vartheta_{d_0+1} \le \cdots \le \vartheta_{\min}.
\]
When $\qmod(g_i^{(k)})$ is archimedean, we have
$\vartheta_d \to \vartheta_{\min}$ as $d \to \infty$,
as shown in \cite{Las01}.
If $\vartheta_d = \vartheta_{\min}$ for some $d$,
the relaxation (\ref{d-mom}) is said to be exact (or tight)
for solving (\ref{subopt}). For such a case,
the Moment-SOS hierarchy is said to have finite convergence.
The Moment-SOS hierarchy has finite convergence
when the archimedean and some optimality conditions hold~\cite{nie2014optimality}.
Although there exist special polynomials such that
the Moment-SOS hierarchy fails to have finite convergence,
such special problems belong to a set of measure zero in the space of
input polynomials \cite{nie2014optimality}.

\subsection{Some properties of Algorithm~\ref{gs}}

Although Algorithm~\ref{gs} converges for many problems,
it is possible that it does not converge for some special ones.
For instance, it is possible that (\ref{subopt}) becomes infeasible
after some loops even if the starting point $\allx^{(0)}$ is feasible.
The following is such an example.

\begin{example}
\rm  \label{nonfeasible}
Consider the $2$-player GNEP
\begin{equation}
\begin{array}{ccccc}
\min\limits_{x_1 \in \re^1 } & -x_1-x_2 & \vline &
\min\limits_{x_2 \in \re^1} & x_1 x_2 \\
\st & 0\le x_1\le 2&\vline&s.t.&x_1+(x_2)^2 \le 1  .
\end{array}
\end{equation}
If Algorithm~\ref{gs} begins with $(x_1^{(0)}, x_2^{(0)})=(0,1)$
and uses the constant $\tau^{(k)} =0.05$,
then $x_1^{(1)} = 2$ and (\ref{subopt}) is infeasible for $k=1$ and $i=2$.
\end{example}

When a GNEP has a shared constraint, i.e.,
there exists a set $X \subseteq \nfR^n$ such that
$\Xpi(\xmi)=\{\xpi:(\xpi,\xmi)\in X\}$
for all players, then the suboptimization (\ref{subopt}) is feasible for all $k$,
provided that the initial point $\allx^{(0)}$ is feasible \cite{Facchinei2011}.
Beyond the concern of infeasibility, the sequence of $\allx^{(k)}$
produced by Algorithm~\ref{gs} might be alternating and does not converge.
Let's see the following example.

\begin{example}
\rm \label{alterep}
Consider the $2$-player GNEP
\be
\begin{array}{ccccc}
\min\limits_{ x_{1} \in \re^1 }& x_1 & \vline & \min\limits_{ x_2 \in \re^1 } & x_1 x_2\\
\st & x_1\ge x_2 &\vline& \st & (x_1)^2+(x_2)^2=2 .
\end{array}
\ee
If Algorithm~\ref{gs} starts with $(x_1^{(0)},x_2^{(0)})=(1,1)$
and uses the constant $\tau^{(k)} =0.001$,
the sub-optimization (\ref{subopt}) for the first player is
\[
\left\{ \begin{array}{rl}
\min\limits_{x_{1} \in \re^1 }&x_1+0.001(x_1-1)^2\\
\st & x_1\ge 1 .
\end{array} \right.
\]
Its minimizer $x_1^{(1)}=1$. After plugging $(x_1^{(1)},x_2^{(0)})$
into (\ref{subopt}), the sub-optimization (\ref{subopt}) for the second player is
\[
\left\{ \begin{array}{rl}
\min\limits_{ x_{2} \in \re^1 }& x_2+0.001(x_2-1)^2 \\
\st & x_2^2=1 ,
\end{array} \right.
\]
whose minimizer $x_2^{(1)} = -1$.
After one iteration, Algorithm~\ref{gs} produced the point
$\allx^{(1)} = (1,-1)$. For the loop of $k=1$,
the sub-optimization problem (\ref{subopt}) for the first player is
\[
\left\{ \begin{array}{rl}
\min\limits_{x_{1}\in \re^1} & x_1+0.001(x_1-1)^2 \\
\st & x_1\ge -1 ,
\end{array} \right.
\]
whose minimizer $x_1^{(2)}=-1$,
and the sub-optimization (\ref{subopt}) for the second player is
\[
\left\{ \begin{array}{rl}
\min\limits_{x_{2}\in \re^1} & -x_2+0.001(x_2+1)^2\\
\st & x_2^2=1,
\end{array} \right.
\]
whose minimizer $x_2^{(2)}=1$. So,  $(x_1^{(2)},x_2^{(2)})=(-1,1)$.
Continuing this process,
one can show that $\allx^{(k)}$ is alternating in the pattern
\[
(1,1)\longrightarrow(1,-1)\longrightarrow(-1,-1)\longrightarrow(-1,1)
\longrightarrow (1,1) \longrightarrow \cdots.
\]
Algorithm \ref{gs} does not converge for this GNEP.
\end{example}

We would like to remark that even for the case that
Algorithm~\ref{gs} converges, the limit of $\allx^{(k)}$
is not necessarily a GNE for \reff{GNEP}.
This is shown in the following example.

\begin{example}
\label{ep:convnotsol}
\rm
Consider the GNEP in \reff{convnotsol}.
For the first player, when $x_2>0$, its feasible set is
$x_1\ge x_2+1$, so the sub-optimization \reff{subopt} in the $k$th loop is
\be \label{p1subopt}
\left\{ \begin{array}{cc}
\min\limits_{x_{1} \in \re^1 } & x_1 + \tau^{(k)} (x_1-x_1^{(k)})^2\\
\st & x_1\ge1+x_2^{(k)}.
\end{array} \right.
\ee
For $0 < \tau^{(k)} <0.5$, the minimizer of (\ref{p1subopt}) is $1+x_2^{(k)}$.
For the second player, the sub-optimization (\ref{subopt}) in the $k$th loop  is
\be
    \label{p2subopt}
\left\{ \begin{array}{rl}
      \min\limits_{x_{2} \in \re^1 } & (x_2)^2-x_2x_2^{(k)}
                      +\tau^{(k)}(x_2-x_2^{(k)})^2 \\
      \st & (x_2)^2\le 3-(x_2^{(k)}+1)^2, \, x_2 \ge 0.
    \end{array} \right.
\ee
When (\ref{p2subopt}) is feasible, its minimizer is
\[
\min\left\{\frac{1+2\tau^{(k)}}{2+2\tau^{(k)}}x_2^{(k)},
\sqrt{3-(x_2^{(k)}+1)^2} \right\}.
\]
Therefore, for any constant $0 < \tau^{(k)} <0.5$ or a decreasing
$\tau^{(k)}$ with $\tau^{(0)} < 0.5$, if $0<x_2^{(0)}\le \sqrt{3}-1$
(to make (\ref{p2subopt}) feasible),
then $\allx^{(k)} \to (1,0)$ as $k\to\infty$.
However,  $(1,0)$ is not a GNE, because when $x_2=0$,
$x_1=0$ is feasible and it is the minimizer.
Indeed, $(0,0)$ is a GNE. 
This shows that a limit point produced by Algorithm~\ref{gs}
is not necessarily a GNE.
\end{example}

In practice, however, the performance of Algorithm~\ref{gs} is good.
%is not that bad.
Under certain conditions, Algorithm~\ref{gs} converges
and the limit is a GNE. This requires some assumptions
on the feasible sets of \reff{subprob}.
Let $G$ be a set-valued map defined on a set $U$, i.e.,
$G(x)$ is a subset of a range $Y$, for all $x \in U$.
Its domain, $\mbox{dom}\,G$,  is the set of $x\in U$ such that
$G(x)\ne\emptyset$ \cite{aubin2009set}.
The map $G$ is said to be
{\it inner semicontinuous} at $x\in U$ relative to $\mbox{dom}\, G$
if for all $y\in G(x)$ and
for all sequences $\{x_\ell\}\subseteq \mbox{dom}\, G$
such that that $ x_\ell \to x$, there exists  a sequence of
$y_\ell \in G(x_{\ell} )$ converging to $y$.
The map $G$ is called inner semicontinuous
relative to $\mbox{dom}\, G$ if it is inner semicontinuous
relative to $\mbox{dom}\, G$ at every point in $\mbox{dom}\, G$.
For instance, %in the GNEP \reff{GNEP},
if the set $\Xpi(x_{-i}) = \{\xpi:(\xpi,\xmi)\in C_i \}$
for $C_i\subseteq \re^n$ being a polyhedron or a ball,
then the set-valued map $x_{-i} \mapsto \Xpi(x_{-i})$
is inner semicontinuous relative to its domain at all points $x_{-i}$
\cite{rockafellar2009variational}.
However, for the GNEP in \reff{convnotsol},
the set-valued map $x_2\to X_1(x_2)$ is not inner semicontinuous at $(0,0)$
(see the end of this section).
We refer to \cite{Bbank1982,rockafellar2009variational}
for the inner semicontinuity of set-valued maps.
The following is a useful lemma about inner semicontinuity.

\begin{lemma}
\label{proconv}
For two closed sets $U$ and $V$, let $f: U \times V \to \re$
and $h: U \times V \to \re^m$ be two continuous functions.
For $y \in V$, define the set-valued map
\[
G(y)=\{x \in U: h(x,y)\ge0\}.
\]
Consider two sequences $\{ x^{(k)} \} \subseteq \mbox{dom}\,G$
and $\{ y^{(k)} \} \subseteq V$ such that
$x^{(k)} \to x^*$ and $y^{(k)} \to y^*$.
Suppose $0\le\tau^{(k)} \to 0$ as $k \to \infty$.
Assume that each $x^{(k)}$ is a minimizer of the optimization problem
\begin{equation}
  \label{kpro}
\left\{ \begin{array}{ll}
\min\limits_{x \in U} & f(x,y^{(k)})+\tau^{(k)}\Vert x-x^{(k-1)}\Vert^2, \\
\st & h(x,y^{(k)})\ge 0 .
\end{array} \right.
\end{equation}
If the set-valued map $G(y)$ is inner semicontinuous relative to $\dom G$,
then $x^*$ is also a minimizer of
\begin{equation}
\label{limpro}
\left\{ \begin{array}{ll}
\min\limits_{x \in U} & f(x,y^*) \\
\st & h(x,y^*) \ge 0 .
\end{array} \right.
\end{equation}
\end{lemma}
\begin{proof}
We prove it by a contradiction argument.
Suppose otherwise that $x^*$ is not a minimizer of
\reff{limpro}, then there exists $z^* \in G(y^*)$ such that
\be  \label{fz*<fx*}
 f(z^*, y^*) < f(x^*, y^*).
\ee
Since the mapping $ G $ is  inner semicontinuous,
there exists a sequence of $z^{(k)}$ such that
$z^{(k)} \rightarrow z^*$  and  $ z^{(k)} \in G(y^{(k)}).$
The sequence $ \{z^{(k)}\}$ is clearly bounded.
Because $ x_{k} $ is a minimizer of
\begin{equation*}
  \begin{array}{ll}
 %\nonumber to remove numbering (before each equation)
  \min & f(x, y^{(k)})+\tau^{(k)} \|x-x^{(k-1)} \|^2 \\
  s. t. & x \in G(y^{(k)}),
\end{array}
\end{equation*}
we have that
\begin{equation}\label{1}
   f(z^{(k)}, y^{(k)})+\tau^{(k)} \|z^{(k)}-x^{(k-1)} \|^2  \geq
   f(x^{(k)}, y^{(k)})+\tau^{(k)} \|x^{(k)}-x^{(k-1)} \|^2.
\end{equation}
Because $ f(x, y) $ is continuous, it holds that
\[
f(x^{(k)}, y^{(k)})\rightarrow f(x^*, y^*), \quad
f(z_{k}, y_{k})\rightarrow f(z^*, y^*)
\]
as $k \to \infty$. For all $\varepsilon >0$,
there exists $K_{1}$ such that
\begin{equation*}
  \begin{array}{ccc}
  f(x^{(k)}, y^{(k)})-f(x^*, y^*) & > & -\frac{\varepsilon}{4}, \\
  f(z^{(k)}, y^{(k)})-f(z^*, y^*) & < & \frac{\varepsilon}{4}
  \end{array}
  \end{equation*}
for all $ k > K_{1}.$ Combining the two inequalities, we can  get
\[
 f(x^{(k)}, y^{(k)})- f(z^{(k)}, y^{(k)})+f(z^*, y^*)-f(x^*, y^*)  >  -\frac{\varepsilon}{2}.
\]
Therefore, we have
\begin{multline}\label{2}
f(z^*, y^*)-f(x^*, y^*) +\frac{\varepsilon}{2} >f(z^{(k)}, y^{(k)})-f(x^{(k)}, y^{(k)}) \\
\ge \tau^{(k)} \Big( \|x^{(k)}-x^{(k-1)} \|^2 - \|z^{(k)}-x^{(k-1)} \|^2 \Big).
\end{multline}
The last inequality follows from (\ref{1}).
Because $ \{x^{(k)} \} $, $ \{ z^{(k)} \}$ are convergent sequences
and $\tau^{(k)} \to 0 $, there must exist $ K_{2}$ such that
\[
\tau^{(k)}( \|x^{(k)}-x^{(k-1)} \|^2 - \|z^{(k)}-x^{(k-1)} \|^2)
>  -\frac{\varepsilon}{2}
\]
whenever $ k> K_{2}.$  Let $K :=\max\{K_{1}, K_{2}\}$, then for all $ k>K$
\[
    f(z^*, y^*)-f(x^*, y^*) +\varepsilon > 0.
\]
Since $ \varepsilon $ can be arbitrarily small, the above implies that
\[
f(z^*, y^*)-f(x^*, y^*) \geq  0,
\]
which contradicts \reff{fz*<fx*}.
Therefore, $x^*$ is a minimizer of \reff{limpro}.
\end{proof}

Lemma~\ref{proconv} immediately implies the following result.

\begin{theorem}  \label{th:limissol}
Let $\allx^{(k)}$ be the sequence produced by Algorithm \ref{gs}
for the GNEP of \reff{GNEP}. Assume that
$\allx^{(k)} \to \allx^*$ and $\tau^{(k)}\to 0$.
If for each $i$ the set-valued map
$G_i: x_{-i} \mapsto X_i( x_{-i} )$ is inner semicontinuous relative to its domain
$\dom G_i$, then the limit $\allx^*$ is a GNE for the GNEP of \reff{GNEP}.
\end{theorem}

\begin{remark}
Theorem~\ref{th:limissol} assumes that the sequence of $\allx^{(k)}$
produced by Algorithm~\ref{gs} converges. However,
the theorem does not give a sufficient condition for this sequence to converge.
To ensure convergence, we need to assume the GNEPs are GPGs;
see Theorems~\ref{gpgconv1} and \ref{gpgconv2}.
There exists a convergence result \cite[Lemma 1]{Sagratella2017algorithms}
that is similar to Lemma~\ref{proconv} and Theorem~\ref{th:limissol}.
\end{remark}

In the proof of Lemma \ref{proconv}, it is required that
$\tau^{(k)}\to 0$, which is also assumed in Theorem \ref{th:limissol}.
However, in the implementation of Algorithm \ref{gs},
we do not need $\tau^{(k)}\to 0$. Sometimes, a constant
$\tau^{(k)}$ works very well.
We refer to Theorem~\ref{gpgconv2} and examples in Section~\ref{nr}.

For Examples~\ref{nonfeasible} and \ref{alterep},
Algorithm~\ref{gs} does not produce a convergent sequence.
For Example~\ref{ep:convnotsol}, the set-valued map $G_1:x_2\mapsto X_1(x_2)$
for the first player is not inner semicontinuous relative to its domain
$\dom G_1$. In fact, at the point $(x_1,x_2)=(0,0)$,
it is clear that $x_1\in G_1(x_2)$.
However, for every sequence $\{x_2^{(k)}\}$ such that
$0<x_2^{(k)}\to x_2=0$, $G_1(x_2^{(k)})=[x_2^{(k)}+1,\infty).$
Since each $x_2^{(k)}>0$, there does not exist a sequence
$\{x_1^{(k)}\}$ converging to $x_1=0$ and
$x_1^{(k)}\in G_1(x_2^{(k)})=[x_2^{(k)}+1,\infty)$.
Therefore, the inner semicontinuity assumption in Theorem~\ref{th:limissol}
fails for Example~\ref{ep:convnotsol}.

\section{Generalized potential games}
\label{sc:GPG}

The Gauss-Seidel method is frequently used for solving GNEPs.
However, its convergence is not guaranteed for all of them.
One wonders for what kind of GNEPs the Gauss-Seidel method converges.
The generalized potential game (GPG) is such a GNEP.
The following is the definition of GPGs in \cite{Facchinei2011}.

\begin{definition}
(\cite{Facchinei2011}) \label{defgpg}
The GNEP of \reff{GNEP} is a generalized potential game if:
\begin{enumerate}

\item [(i)]
\label{df:GPGitem1}
There exists a closed set $\emptyset \ne X \subseteq \mathbb{R}^n$ such that
\[ \Xpi(\xmi)\equiv\{\xpi\in D_{i}:(\xpi,\xmi)\in X\} \]
for all players, where each $D_{i}\subseteq \mathbb{R}^{n_i}$
is a closed set such that $(D_1 \times \cdots \times D_N) \cap X\ne \emptyset$.

\item [(ii)] There exist a continuous function $P(x):\mathbb{R}^n \to\mathbb{R}$
and a forcing function $\sigma:\mathbb{R}_{+}\to\mathbb{R}_+$
(i.e., $\lim_{k\to\infty}\sigma(t_k)=0$ implies $\lim_{k\to\infty}t_k=0$)
such that for all $y_{i},x_{i}\in X_i(\xmi)$
\be
\label{GPGcon2}
  \baray{c}
\fpi(y_i,\xmi)-\fpi(x_i,\xmi)>0 \quad  \Longrightarrow \\
P(y_{i},\xmi)-P(x_{i},\xmi)\ge\sigma(\fpi(y_{i},\xmi)-\fpi(x_{i},\xmi)).
\earay
\ee
\end{enumerate}
\end{definition}

The item (i) in Definition \ref{df:GPGitem1}
is from the concept of {\it shared constraint} \cite{Facchinei2007}.
It implies that if $\allx^{(0)}$ is feasible,
then the sub-optimization problem (\ref{subopt})
is feasible for all $k$ and $i$.
The item (ii) means that there exists a single ``dominant function"
$P$ that measures the changes on
each player's objective functions \cite{Facchinei2011}.

Some special GNEPs can be directly verified as GPGs.
For instance, for the GNEP of \reff{GNEP},
if each objective $f_i$ can be expressed as
\be\label{eq:classicgpg}
f_i(\allx) = f_0(\allx)+\sum_{j=1}^M f_{i,j}(x_j)
\ee
for some functions $f_0$ and $f_{i,j}$ and the item (i) holds,
then the GNEP of \reff{GNEP} is a GPG because $P, \sigma$ can be chosen as
\[
P(\allx) = f_0(\allx) + \sum_{i=1}^N \sum_{j=1}^M f_{i,j}(x_j),
\quad \sigma(t)=t.
\]
One can easily check that the above $P(x)$ and $\sigma(t)$ 
satisfy (\ref{GPGcon2}) \cite{Monderer1996}.

GPGs are extensions of potential games, which were originally
introduced for NEPs \cite{Monderer1996}.
They have broad applications \cite{Neel2002}.
The following is an example of GPG arising from applications.

\begin{example}\rm
  \label{gpgpollute}
The GNEPP from the environmental pollution control, described
%in example \ref{pollute}
in the introduction, is a GPG.
The functions $P$ and $\sigma$ can be chosen as
\[
 \begin{array}{rl}
    P(\allx)=& 2\prod\limits_{i=1}^N (x_{i,0}-
     \sum\limits_{j=1}^N\gamma_{j,i}x_{j,i})+
\sum\limits_{i=1}^N\left[\sum\limits_{j=0}^{N}x_{i,j}
-\sum\limits_{j=1}^N\gamma_{j,i}x_{j,i}-x_{i,0}(b_i-1/2x_{i,0})\right], \\
    \sigma(t)=& t.
\end{array}
\]
The numerical results of Algorithm \ref{gs} are shown in the next section.
\end{example}

%Facchinei et al.~\cite{Facchinei2011}
%introduced the definition of generalized potential games for GNEPs.
The following is the convergence result for Algorithm~\ref{gs}
when it is applied to solve GPGs.

\begin{theorem}(\cite[Theorem~5.2]{Facchinei2011})  \label{gpgconv1}
Consider the GNEP of \reff{GNEP} such that all the functions are continuous.
Assume that \reff{GNEP} is a GPG and each set-valued map
$G_i: x_{-i} \mapsto X_i(x_{-i})$ is inner semicontinuous
relative to its domain. In Algorithm \ref{gs},
suppose each $x_i^{(k+1)}$ is a minimizer of \reff{subopt}
and the parameters $\tau^{(k)}$ are updated as
\be \label{update:tau}
\tau^{(k+1)} \, := \, \max \Big\{
\min \Big[\tau^{(k)},\max_{i=1,\dots,N}
(\Vert x_{i}^{(k+1)}-x_{i}^{(k)}\Vert) \Big], 0.1\tau^{(k)} \Big\}.
\ee
Then every limit point of the sequence $\{\allx^{(k)}\}_{k=0}^{\infty}$
produced by Algorithm \ref{gs} is a GNE for \reff{GNEP}.
\end{theorem}

The updating scheme \reff{update:tau} for $\tau^{(k)}$
is a bit complicated. However, if each player's optimization problem \reff{subprob}
is convex, then the parameter $\tau^{(k)}$ can be chosen to be constant.

\begin{theorem}(\cite[Theorem~4.3]{Facchinei2011}) \label{gpgconv2}
Consider the GNEP of \reff{GNEP} such that all the functions are continuous.
Assume that \reff{GNEP} is a GPG and each set-valued map
$G_i: x_{-i} \mapsto X_i(x_{-i})$ is inner semicontinuous
relative to its domain. Suppose the objectives $f_i(\,\cdot\,,x_{-i})$
and the feasible sets $X_i(x_{-i})$ are all convex.
In Algorithm \ref{gs}, suppose each $x_i^{(k+1)}$ is a minimizer of \reff{subopt}
and the parameter $\tau^{(k)}=\tau>0$ is a constant.
Then every limit point of the sequence $\{\allx^{(k)}\}_{k=0}^{\infty}$
produced by Algorithm \ref{gs} is a GNE for \reff{GNEP}.
\end{theorem}

%Besides GPGs,
Beyond GPGs, the Gauss-Seidel method has convergence for GNEPs
with discrete strategy sets\cite{Sagratella2016computing} or mixed-integer variables\cite{Sagratella2017computing}.
In general, when \reff{GNEP} is not a GPG, the convergence of Algorithm~\ref{gs}
is not known very much. We have seen examples in Section~\ref{sc:GSAlg}
such that Algorithm~\ref{gs} fails to converge.
On the other hand, the performance of Algorithm~\ref{gs}
is actually very good in our computational experiments (see Section~\ref{nr}).
In the following, we discuss how to certify that a GNEP is a GPG.

\subsection{A certificate for GPGs}
\label{gpgvrf}

Generally, it is hard to check whether a GNEP is a GPG or not.
The main challenge is to verify the item (ii) in Definition~\ref{defgpg}.
In this subsection, we give a certificate for \reff{GPGcon2} to hold.
For the $i$th player, denote the set
\be \label{set:Ki}
K_{i} = \left\{(\xpi,\ypi,\xmi) \in \re^{n_i} \times \re^{n_i} \times \re^{n-n_i}
\left| \baray{c}
\xpi,\ypi\in\lXpi \\ \fpi(\ypi,\xmi)-\fpi(\xpi,\xmi) \ge 0
\earay \right.
\right\}.
\ee
For convenience, denote the differences of functions
\be
\left\{ \baray{rcl}
\triangle P_i &:=& P(\ypi,\xmi)-P(\xpi,\xmi), \\
\triangle f_i &:=& f_{i}(\ypi,\xmi)-f_{i}(\xpi,\xmi).
\earay \right.
\ee
The following lemma is straightforward for verification.

\begin{lemma} \label{lm:cert:GPG}
For the GNEP of \reff{GNEP}, if the item~(i) in Definition~\ref{defgpg}
holds, and there exist polynomials $P \in \re[\allx]$,
$p_{i,0},p_{i,1} \in \re[\xpi,\ypi,\xmi]$ ($i=1,\ldots,N$)
such that $p_{i,0} \ge 0, \, p_{i,1} \ge 0$ on $K_{i}$ and
\be
\label{nonnegarep}
\triangle P_i \,=\, (p_{i,0} + 1 ) \triangle f_i + p_{i,1}
\ee
for all $i$, then \reff{GNEP} is a GPG.
\end{lemma}

In the equation (\ref{nonnegarep}), we can replace the constant $1$
by any positive number $\epsilon>0$, up to scaling coefficients.
For numerical reasons, we prefer the constant $1$.
Lemma~\ref{lm:cert:GPG} gives a certificate for GPGs.
The following are examples of GPGs certified by \reff{nonnegarep}.

\begin{example}
\rm \label{puregpg1}
Consider the $2$-player GNEP with the sets
\[
X=\{(x_1,x_2):1\le x_{1},x_{2}\le10,\ x_{1}\ge x_{2}\},
\]
\[
X_1(x_{-1})=\{x_1:(x_1,x_2)\in X\}, \quad
X_2(x_{-2})=\{x_2:(x_1,x_2)\in X\}.
\]
The two players' optimization problems are respectively
\be
\label{nonnegarepep1}
\begin{array}{ccccc}
\min\limits_{x_{1}\in X_1} & x_{1}+x_{2} & \vline & \min\limits_{x_{2}\in X_2} &-x_{1}x_{2}.
\end{array}
\ee
Let $P(x_{1},x_{2})=(x_{1})^3-x_{1}x_{2}+x_{1}$, we have
\begin{equation}
\left\{  \begin{array}{ll}
    \triangle P_1&=(y_{1}-x_{1})[(y_{1}-x_{1})^2+1]+
    (3y_{1}x_{1}-x_{2})(y_{1}-x_{1}), \\
    \triangle P_2&=-x_{1}(y_{2}-x_{2}), \\
    \triangle f_1&=y_1-x_1, \\
    \triangle f_2&=-x_1(y_2-x_2). \\
\end{array} \right.
\end{equation}
The equation \reff{nonnegarep} is satisfied for
\[
p_{1,0}=(y_{1}-x_{1})^2, \, p_{1,1}=(3y_{1}x_{1}-x_{2})(y_{1}-x_{1}), \,
p_{2,0} = p_{2,1} = 0.
\]
It is clear that $p_{1,0},p_{2,0},p_{2,1}$ are nonnegative.
By the definition of $K_{1}$,
$\triangle f_1\ge0$, and $3y_1x_1-x_2\ge 3y_1-x_2\ge0$,
%so $p_{1,1}$ is nonnegative on $K_{1}$.
so $p_{1,1} \ge 0$ on $K_{1}$.
\end{example}

\begin{example}
\rm \label{puregpg2}
Consider the $2$-player GNEP with the sets
\[
X=\{(x_1,x_2):(x_1)^3+(x_2)^3\le 2, x_1\ge6x_2\}, \,
\]
\[
X_1(x_{-1})=\{x_1:(x_1,x_2)\in X\}, \quad
X_2(x_{-2})=\{x_2:(x_1,x_2)\in X\}.
\]
The two players' optimization problems are respectively
\be
\label{nonnegarepep2}
\begin{array}{lllll}
\min\limits_{x_{1}\in X_1} & (x_{1})^2x_{2}+(x_{2})^2x_{1}-4(x_1)^4 & \vline
& \min\limits_{x_{2}\in X_2} & x_{1}x_{2}-3(x_2)^2 \\
\st & x_{1} \ge 0, & \vline & \st & x_{2} \ge 0.125 .\\
\end{array}
\ee
For $P(x_{1},x_{2})=(x_{1})^2x_{2}+(x_{2})^2x_{1}-4(x_1)^4$,
\begin{equation}
\left\{   \begin{array}{rcl}
 \triangle P_1&=&\triangle f_1=(y_{1})^2x_{2}+(y_{2})^2x_{1}-4(y_1)^4 \\
              & & \qquad -(x_{1})^2x_{2}-(x_{2})^2x_{1}+4(x_1)^4, \\
 \triangle P_2 &=& (x_1)^2(y_2-x_2)+x_1((y_2)^2-(x_2)^2), \\
 \triangle f_2 &=& x_1(y_2-x_2)-3(y_2)^2+3(x_2)^2. \\
\end{array} \right.
\end{equation}
The equation (\ref{nonnegarep}) holds with
\[
p_{1,0}=0, \, p_{1,1}=0, p_{2,0}=x_1,
p_{2,1}=(y_2-x_2)[4x_1(y_2+x_2)+3(y_2+x_2)-x_1].
\]
Clearly, $p_{1,0},\, p_{1,1}, \,p_{2,0}\ge0$. Note that
\[
\triangle f_2=(y_2-x_2)(x_1-3(y_2+x_2))  \ge  0
\]
on $K_2$. Then, either $x_1-3(y_2+x_2)>0$ hence $y_2-x_2\ge0$,
or $x_1-3(y_2+x_2)=0$, which forces $y_2-x_2=0$.
This is because $x_1\ge 6y_2$, $x_1\ge 6x_2$ and $y_2,x_2>0$,
if $x_1-3(y_2+x_2)=0$, then the only possible case is $x_1=6y_2=6x_2$.
Thus from $$4x_1(y_2+x_2)+3(y_2+x_2)-x_1> x_1(4y_2+4x_2-1)\ge0,$$
we know $p_{2,1}\ge0$ on $K_2$.
\end{example}

\begin{example}
\rm \label{puregpg3}
Consider the $2$-player GNEP with the sets
\[
X = \left\{(x_1, x_2)
\left| \baray{l}
x_1=(x_{1,1}, x_{1,2}) \in \nfR^2, x_2\in \nfR, \\
x_{1,1}, x_{1,2}, x_2\ge 0.5, \\
x_2-0.3\le x_{1,1}+x_{1,2}\le x_2+0.3
\earay \right.
\right\},
\]
and $X_1(x_{-1})=\{x_1:(x_1,x_2)\in X\}$,
$X_2(x_{-2})=\{x_2:(x_1,x_2)\in X\}$.
The optimization problems are respectively
\be
\label{nonnegarepep3}
\begin{array}{lllll}
\min\limits_{x_{1}\in X_1} & x_{1,1}x_2+x_{1,2}x_2 & \vline &
\min\limits_{x_{2}\in X_2} & x_{1,1} \cdot x_{1,2} \cdot x_2 \\
\st & \Vert x_1\Vert =2 & \vline & &
\end{array}
\ee
For $P(x_{1}, x_{2})=(x_{1,1}+x_{1,2}+1)^3 x_2$,
\begin{equation}
\left\{   \begin{array}{ll}
    \triangle P_1&=x_2((y_{1,1}+y_{1,2}+1)^3-(x_{1,1}+x_{1,2}+1)^3), \\
    \triangle P_2&=(y_2-x_2)(x_{1,1}+x_{1,2}+1)^3, \\
    \triangle f_1&=x_2(y_{1,1}+y_{1,2}-x_{1,1}-x_{1,2}), \\
    \triangle f_2&=x_{1,1}x_{1,2}(y_2-z_2).
\end{array} \right.
\end{equation}
The equation (\ref{nonnegarep}) holds with
\begin{eqnarray*}
   p_{1,0}&=&(y_{1,1}+y_{1,2}+1)^2+(x_{1,1}+x_{1,2}+1)^2+(y_{1,1}+y_{1,2})(x_{1,1}+x_{1,2}) \\
    && \qquad +y_{1,1}+y_{1,2} +x_{1,1}+x_{1,2}, \\
    p_{1,1}&=&0, \\
    p_{2,0}&=&3x_{1,1}+3x_{1,2}+5, \\
    p_{2,1}&=&(y_2-x_2)\big[(x_{1,1})^3+(x_{1,2})^3+3(x_{1,1})^2+ \\
    && \qquad 3(x_{1,2})^2+3x_{1,1}+3x_{1,2}+1 \big].
\end{eqnarray*}
The equality $\triangle P_1=(1+p_{1,0})\triangle f_1$ follows from the identity
\[
a^3-b^3 \, = \, (a-b)(a^2+ab+b^2)
\]
with $a=y_{1,1}+y_{1,2}+1$ and $b=x_{1,1}+x_{1,2}+1$.
Clearly, $p_{1,0},p_{1,1} \geq 0$ on $K_1$, and $p_{2,0} \geq 0$ on $K_2$.
Since $x_{1,1}x_{1,2}>0$,
$\triangle f_2\ge0$ implies $y_2-x_2\ge0$,
so $p_{2,1}\ge0$ on $K_2$.
\end{example}

\subsection{Putinar Positivstellensatz for the certificate}

Lemma~\ref{lm:cert:GPG} gives a convenient certificate for checking GPGs.
One needs to find polynomials $p_{i,0},p_{i,1}$ and $P$ satisfying (\ref{nonnegarep})
and $p_{i,0},p_{i,1} \ge 0$ on $K_i$.
For a polynomial tuple $h$, we have seen that if $p \in \qmod(h)$,
then $p \ge 0$ on the semialgebraic set $\mathcal{S}(h)$.
This motivates us to use Putinar's Positivstellensatz
%for getting appropriate certificates.
for verifying that.

For the set $K_i$ as in \reff{set:Ki},
let $h_i:=(h_{i,t})_{t=1}^{m_i}$ be a tuple of polynomials in
$\re[x_i, y_i, x_{-i}]$ such that
\[
K_i = \{(x_i, y_i, x_{-i}): h_i(x_i, y_i, x_{-i}) \ge 0 \}.
\]
Moreover, let $h_{i,0}=1$ for all $i$.
When the item (i) in Definition~\ref{defgpg} holds,
the GNEP of \reff{GNEP} is a GPG if there exist
$P\in\nfR[\allx]$ and $q_{i,0}, q_{i,1}\in \qmod(h_i)$
such that
\begin{equation}
  \label{sosrep}
  \triangle P_i=(q_{i,0}+1)\triangle f_i+q_{i,1}
\end{equation}
for all players. For an even degree $2d$,
we parameterize $P, q_{i,0}, q_{i,1}$ as
\[
P(\allx) = \mathbf{p}^T [\allx]_{2d}, \quad
q_{i,0} = {\sum}_{t=0}^{m_i} \big( [\allx,y_i]_{d-d_{it} } \big) ^T
\cdot Q_{i,0}^t \cdot \big( [\allx,y_i]_{d-d_{it} } \big)\cdot h_{i,t},
\]
\[
q_{i,1} = {\sum}_{t=0}^{m_i} \big( [\allx,y_i]_{d-d_{it} } \big) ^T
\cdot Q_{i,1}^t \cdot \big( [\allx,y_i]_{d-d_{it} } \big)\cdot h_{i,t}.
\]
In the above, the degree
$d_{it} = \lceil \deg(h_{i,t})/2 \rceil.$
One can show that $q_{i,0}, q_{i,1}\in \qmod(h_i)$ if and only if
there exist psd matrices $Q_{i,0}^t, Q_{i,1}^t$
in the above parametrization, for some $d$
\cite[Chapter 2]{lasserre2015introduction}.
For notational convenience,  denote
\be \label{def:qQ}
% q \, := \, \big( q_{i,0},  q_{i,1} \big)_{ i =1, \ldots, N}, \quad
Q \, := \, \big( Q_{i,0}^t, Q_{i,1}^t \big)_{ i =1, \ldots, N, t = 1, \ldots, m_i }.
\ee
Therefore, the certificate \reff{nonnegarep} in Lemma~\ref{lm:cert:GPG}
can be checked by solving the semidefinite program
\be  \label{sosrlx}
\left \{ \begin{array}{rl}
\min\limits_{\mathbf{p},Q}  & \sum\limits_{i,t} \mbox{trace} \Big( Q_{i,0}^t + Q_{i,1}^t \Big)  \\
\st &\triangle P_i\equiv(q_{i,0}+1)\triangle f_i+q_{i,1} \,(\forall \, i), \\
  & P \in\nfR[\allx]_{2d}, \\
  &  Q_{i,0}^t \succeq 0, \, Q_{i,1}^t \succeq 0  \, (\forall \, i, t).
\end{array} \right.
\ee

The certificate given by solving (\ref{sosrlx}) does not require to
have priori polynomials $P,q_{i,0}$ and $q_{i,1}$.
Instead, the coefficients of these polynomials are variables in (\ref{sosrlx})
that are awaiting to be solved numerically.

\begin{example}\rm \label{ep:5.4}
Consider the $2$-player GNEP such that the two players' optimization problems are
\begin{equation}
  \begin{array}{cllcl}
    \min\limits_{x_1 \in \re^1} & 2x_2-x_1& \vline &
    \min\limits_{x_2 \in \re^1 }& (x_1)^2-2x_1x_2-(x_2)^2 \\
    \st &(x_1)^2+(x_2)^2\le 1, &\vline& \st &(x_1)^2+(x_2)^2 \le 1, \\
    & x_1 \ge 0, &\vline& & x_2\ge 0.
    \end{array}
\end{equation}
The $\Delta f_1=x_1-y_1$, $\Delta f_2=2x_1(x_2-y_2)+(x_2)^2-(y_2)^2$,
and the defining polynomial tuples $h_1,h_2$
for the sets $K_1, K_2$ are respectively,
\[
\begin{array}{l}
h_1=\{1-(x_1)^2-(x_2)^2,\,1-(y_1)^2-(x_2)^2,\, x_1,\, y_1,\,  \Delta f_1\},  \\
h_2=\{1-(x_1)^2-(x_2)^2,\,1-(x_1)^2-(y_2)^2,\, x_2,\, y_2,\,  \Delta f_2\}.
\end{array}
\]
For the degree $d=2$, the semidefinite program \reff{sosrlx} becomes
\be
\label{eq:explicitcertify}
\left \{ \begin{array}{rl}
\min\limits_{\mathbf{p},Q}  & \sum\limits_{i=1}^2\sum\limits_{t=0}^6 \mbox{trace}
\Big( Q_{i,0}^t + Q_{i,1}^t \Big)  \\
\st &\mathbf{p}^T([y_1,x_2]_4-[\allx]_4) =
\left(\sum\limits_{t=0}^6[\allx,y_1]^T_{2-d_{it}}Q_{1,0}^t[\allx,y_1]_{2-d_{it}} h_{1,t}
+1 \right)\cdot \Delta f_1  \\
&\qquad\qquad\qquad\qquad\qquad\qquad +\sum\limits_{t=0}^6[\allx,y_1]^T_{2-d_{it}}Q_{1,1}^t[\allx,y_1]_{2-d_{it}} h_{1,t}, \\
& \mathbf{p}^T([x_1,y_2]_4-[\allx]_4) =
\left(\sum\limits_{t=0}^6[\allx,y_2]^T_{2-d_{it}}Q_{2,0}^t[\allx,y_2]_{2-d_{it}} h_{2,t}
+1 \right)\cdot\Delta f_2 \\
&\qquad\qquad\qquad\qquad\qquad\qquad +\sum\limits_{t=0}^6[\allx,y_2]^T_{2-d_{it}}Q_{2,1}^t[\allx,y_2]_{2-d_{it}} h_{2,t}, \\
  &  \mathbf{p}\in\mathbb{R}^{\mathbb{N}^2_4}\cong\mathbb{R}^{15},\, Q_{i,0}^t \succeq 0, \, Q_{i,1}^t \succeq 0  \, (\forall \, i, t).
\end{array} \right.
\ee
By solving (\ref{eq:explicitcertify}),
we can numerically verify that this GNEP is a GPG.
The computed solution of (\ref{eq:explicitcertify}) is displayed as follows
(the coefficients are displayed with $6$ decimal digits)
\[
\begin{array}{l}
P(\allx)=
% 3.1762x_1-1.0000(x_2)^2-6\cdot10^{-6}(x_1)^2-2.0000x_1x_2+0.2198(x_1)^3-2\cdot10^{-6}(x_1)^4
-3.176290x_1-1.000000(x_2)^2-0.000006(x_1)^2\\
\qquad\qquad\qquad\qquad\qquad\quad-2.000000x_1x_2+0.219805(x_1)^3-0.000002(x_1)^4, \\
q_{10}(\allx,y_1)=2.176289+0.000006x_1+2.000000x_2+0.000006y_1\\
\qquad\qquad\qquad-0.219805(y_1)^2+0.000002(y_1)^3-0.219805(x_1)^2+0.000002(x_1)^3\\
\qquad\qquad\qquad\qquad\qquad\qquad-0.219805x_1y_1+0.000002x_1(y_1)^2+0.000002(x_1)^2y_1, \\
q_{11}(\allx,y_1)=q_{20}(\allx,y_2)=q_{21}(\allx,y_2)=0.
\end{array}
\]
\end{example}

The semidefinite program \reff{sosrlx} is useful for checking GPGs.
For instance, GNEPPs in Example~\ref{ep55} and \ref{ep57} are numerically checked to be GPGs by solving (\ref{sosrlx}).
Moreover, the problems A11-13,A15,A17-A18 in \cite{FacKan10} (see section 5.1)
can be verified to be GPGs in the same way.

\section{Numerical experiments}
\label{nr}

This section reports numerical experiments for solving GNEPPs
by using Algorithm~\ref{gs}. The subproblem \reff{subopt}
is a polynomial optimization problem.
We apply the software {\tt GloptiPoly~3} \cite{GloPol3}
and {\tt SeDuMi} \cite{sturm1999using} to solve
Moment-SOS relaxations of \reff{subopt}.
The semidefinite program (\ref{sosrlx}) for certifying GPGs
is implemented by the software {\tt YALMIP} \cite{yalmip}.
The computation is implemented in a Dell XPS 15 9550 Laptop,
with an Intel\textsuperscript \textregistered \
Core(TM) i7-6700HQ CPU at 2.60GHz$\times $4 and 16GB of RAM,
in a Windows 10 operating system. 
In the computation, the sequence $\{\allx^{(k)}\}_{k=0}^{\infty}$
is regarded to converge if for some $k$ it holds that
\be
  \label{iterdiff}
\Vert \allx^{(i)}-\allx^{(j)}\Vert_{\infty} \, \le \, 10^{-8}\mbox{ for all $i,j \in \{k-10,\ldots,k\}$}.
\ee 
The point $\allx^{(k)}$ is regarded as a GNE
with the accuracy parameter $\varepsilon > 0$ if
\be
\label{def:acc_para}
|f_i(\allx^{(k)})-{f_i}^*|\le\varepsilon
\ee
for all players, where ${f_i}^*$ is the minimum value of (\ref{subprob})
with $x_{-i} = x_{-i}^{(k)}$. 
Our computational results show that Algorithm~\ref{gs}
performs very well for solving GNEPPs, even if for nonconvex ones.
First, we see some examples of the GPGs from Section~\ref{sc:GPG}.

\begin{example}\rm
Consider the environmental pollution problem
%Example~\ref{pollute} and Example~\ref{gpgpollute}.
in the introduction and Example~\ref{gpgpollute}.
We have seen that it is a GPG.
Assume the number of players is $N=2$ and the parameters
$b_1=b_2=2, E_1=E_2=1, \gamma_{1,1}=0.7,\gamma_{1,2}=0.9,\gamma_{2,1}=\gamma_{2,2}=0.8.$
We run Algorithm~\ref{gs}
with $x_{1,0}^{(0)} = x_{1,1}^{(0)} = \cdots = x_{2,2}^{(0)} =0.5$,
and $\tau^{(0)}=0.1$,
$\tau^{(k+1)}$ updated as in (\ref{update:tau}).
After $21$ iterations, we get
\[x_{1,0}^{(21)} =0.9999,\
   x_{1,1}^{(21)}=0,\
   x_{1,2}^{(21)}=0,\
   x_{2,0}^{(21)} =0.7500,\
   x_{2,1}^{(21)}=0,\
   x_{2,2}^{(21)}=0.9375.
\]
Its accuracy parameter $\varepsilon =1.7856\cdot 10^{-8}$.
It costs about $7$ seconds.
\end{example}

\begin{example}  \rm
i) \label{ep51}
Consider the GNEP in Example \ref{puregpg1}. It is a GPG.
All the individual optimization problems are convex.
We run Algorithm~\ref{gs} with the initial point
$(x_1^{(0)}, x_2^{(0)}) = (3, 2)$ and fixed $\tau^{(k)} = 0.02$, and get a GNE $(2.0000, 2.0000)$ with $\varepsilon=6.1541\cdot10^{-8}$.
It runs $12$ iterations and costs  $2.6289$ seconds. \\
ii) \label{ep52}
Consider the GNEP in Example \ref{puregpg2}. It is a GPG.
We run Algorithm~\ref{gs} with the initial point
$(x_1^{(0)}, x_2^{(0)}) = (1, 0.125)$ and fixed $\tau^{(k)} = 0.02$.
It returns the GNE $(1.2595, 0.1250)$ with $\varepsilon=2.2891\cdot10^{-9}$.
It runs $12$ iterations and costs around $2$ seconds.  \\
iii) \label{ep53}
Consider the GNEP in Example \ref{puregpg3}. It is a GPG.
All the individual optimization problems are convex.
We run Algorithm~\ref{gs} with the initial point
$(x_{1,1}^{(0)}, x_{1,2}^{(0)}, x_2^{(0)}) = (1, 1, 2)$ and fixed $\tau^{(k)} = 0.02$.
It returns the GNE $(1.3229, 0.5000, 1.5229)$ with $\varepsilon=1.3631\cdot 10^{-7}$.
It runs $12$ costs around $3$ seconds.\\
iv) \label{ep54}
Consider the GNEP in Example~\ref{ep:5.4}.
It is numerically verified to be a GPG.
We run Algorithm~\ref{gs} with the initial point
$(x_1^{(0)}, x_2^{(0)}) = (0.2, 0.3)$ and fixed $\tau^{(k)} = 0.02$.
For $k=12$, we get $\allx^{(12)} = (0.9539, 0.3)$.
The iteration difference is $2.1792\cdot10^{-8}$
and the GNE accuracy $\varepsilon=5.4170\cdot 10^{-9}$.
It costs about $1.6$ seconds.
\end{example}

\begin{example} \rm
\label{ep55}
Consider the $2$-player GNEP such that
the individual optimization problems are respectively
\begin{equation*}
  \begin{array}{cllcl}
    \min\limits_{x_{1} \in \re^2}&x_{1,1}(x_{1,2}+2x_{2,1}+2x_{2,2})&\vline&\min\limits_{x_{2} \in \re^2}& (x_{1,1})^2+(x_{1,2})^2\\
    & \quad+x_{1,2}(x_{2,1}+x_{2,2})
    +2x_{2,1}x_{2,2}&\vline&&\qquad\quad-(x_{2,1})^2-(x_{2,2})^2\\
    \st &\sum\limits_{i=1}^2\sum\limits_{j=1}^2x_{i,j}=1, &\vline& \st &\sum\limits_{i=1}^2\sum\limits_{j=1}^2x_{i,j} = 1,\\
    &x_{1,1}\ge0,\,x_{1,2} \ge 0,&\vline&&x_{2,1}\ge0,\,x_{2,2} \ge 0.
  \end{array}
\end{equation*}
By solving the semidefinite program (\ref{sosrlx}),
we can numerically check that this GNEP is a GPG.
Run Algorithm~\ref{gs} with the initial points
$x_1^{(0)} = (0.2, 0.3)$, $x_2^{(0)} = (0.2, 0.3)$
and fixed $\tau^{(k)} =0.02$. After $19$ loops, we get that
\[
\allx^{(19)} = (0,  0.5, 0, 0.5).
\]
as a GNE with accuracy parameter $\varepsilon=5.1857\cdot 10^{-7}$.
It costs about $5.36$ seconds.
\end{example}

\begin{example} \rm  \label{ep57}
Consider the $2$-player GNEP whose optimization problems are
\[
\begin{array}{cllcl}
    \min\limits_{x_{1} \in \re^2 }& -2(x_{1,2})^2 + x_{2,1}x_{1,2} + x_{1,1}x_{2,1} &\vline&
    \min\limits_{x_{2} \in \re^2 }& (x_{2,1})^2 - 2x_{1,2}x_{2,2} \\
    &&\vline&
    &\qquad\qquad - 2x_{1,1}x_{2,2} + (x_{2,2})^2\\
    \st & x_{1,1}+x_{1,2}+x_{2,1}+x_{2,2}=1&\vline& \st &x_{1,1}+x_{1,2}+x_{2,1}+x_{2,2}=1\\
    &x_{1,1},x_{1,2}\ge0.1&\vline&&x_{2,1},x_{2,2}\ge0.1
\end{array}
\]
By solving the semidefinite program (\ref{sosrlx}),
one can numerically check that this GNEP is a GPG.
We run Algorithm~\ref{gs} with
\[ \allx^{(0)} = (0.25, 0.25, 0.25, 0.25), \quad \tau^{(0)} =0.1, \]
and $\tau^{(k+1)}$ updated as (\ref{update:tau}).
For $k=12$, we get
\[
\allx^{(12)} = (0.1000, 0.4000, 0.1000, 0.4000),
\]
which is a GNE. The accuracy $\varepsilon=1.14611\cdot 10^{-8}$.
It costs around $2.7$ seconds.
\end{example}

\begin{example} \rm  \label{ep58}
Consider the GNEP whose optimization problems are
\[
 \baray{cllcl}
%\label{gnepgpg1}
 \min\limits_{x_{1} \in \re^2} & (x_{1,1})^2+(x_{1,2})^2+x_{1,1}+x_{1,2}&\vline& \min\limits_{x_{2} \in \re^2} & (x_{2,2})^2-x_{2,1}x_{2,2}\\
\st & \Vert x_1\Vert^2+\Vert x_2\Vert^2\le 1,&\vline&
        \st &\Vert x_1\Vert^2+\Vert x_2\Vert^2\le 1, \\
    & x_{1,1}\ge0,\, x_{1,2}\le 0.5,&\vline& &x_{2,1}\le0,\, 0.3\le x_{2,2}\le 0.8.
\earay
\]
This is a GPG \cite{Facchinei2011}. We run Algorithm~\ref{gs} with
\[
x_1^{(0)} = (0.5,0.5), \quad x_2^{(0)} =(-0.6,0.6), \quad \tau^{(0)} =0.1,
\]
and $\tau^{(k+1)}$ updated as (\ref{update:tau}).
For $k=16$, we get $\allx^{(16)} = (0, -0.5,  0, 0.3) $ as a
GNE with accuracy parameter $\varepsilon =4.1908\cdot 10^{-10}$.
It costs around $4.11$ seconds.
\end{example}

\begin{example} \rm
  \label{needtau}
Consider the $3$-player GNEP whose optimization problems are
\begin{equation*}
% \label{jump}
\begin{array}{cllcllcl}
  \min\limits_{x_1 \in \re^1}& (x_1-x_{2})^2  &\vline&
  \min\limits_{x_2 \in \re^1}& (x_{2}-x_{3})^2 &\vline&
  \min\limits_{x_3 \in \re^1}& (x_{3}-x_1)^2  \\
\st & %(x_1)^2+(x_{2})^2+(x_{3})^2\le 10
    \sum_{i=1}^3 (x_i)^2 \le 10
    &\vline& \st &x_{2}\le3&\vline&
\st & % x_1+x_{2}+x_{3}\le 6 .
    \sum_{i=1}^3 x_i \le 6.
\end{array}
\end{equation*}
Any feasible point $\allx$ with $x_{1}=x_{2}=x_{3}$ is a GNE,
with optimal value $0$ for all players.
If we run Algorithm~\ref{gs} with
$(x_1^{(0)}, x_2^{(0)}, x_3^{(0)}) = (0,1,2)$
and $\tau^{(k)} = 0$ (which is actually not allowed since we require $\tau>0$,
but we still show the result of $\tau=0$
in order to show the necessity of a positive $\tau$),
then we get an alternating sequence
\[
(0,1,2)\longrightarrow(1,1,2)\longrightarrow(1,2,2)\longrightarrow(1,2,1)\longrightarrow(2,2,1)
\]
\[
\longrightarrow(2,1,1)\longrightarrow(2,1,2)\longrightarrow(1,1,2)\longrightarrow \cdots.
\]
If we run Algorithm~\ref{gs} with the same initial point $\allx^{(0)} = (0,1,2)$
but different regularization parameter $\tau^{(k)}$,
the computational results are reported in Table~\ref{needtautable}.
We run it for five different $\tau^{(k)}$.
Two of them are fixed values $0.1,0.05$,
and the other one is $\tau_0=0.5$, $\tau^{(k+1)}$ updated as (\ref{update:tau}).
In the table, ``Iteration Difference" is the value of
\[
\max\limits_{291\le i<j\le300}\Vert \allx^{(i)}-\allx^{(j)}\Vert_{\infty}
\] since none of these five
sequence satisfies (\ref{iterdiff}) at the $300$ iteration.
And ``error" means the number that $\allx^{(300)}$ can be verified up to as a GNE.

\begin{table}[!htbp]
\centering
\caption{Computational Results for Example~\ref{needtau}}
\label{needtautable}%
    \begin{tabular}{|r|r|r|r|c|c|}  \hline
    $\tau^{(k+1)}$      & $x_1$ & $x_2$ & $x_3$ & {Iteration Difference} & $\varepsilon$  \\ \hline
    0.1000   & 1.4289   & 1.4289  &  1.4289 & $1.9139\cdot 10^{-5} $& $10^{-7}$\\
    0.0500  & 1.4494  &  1.4494  &  1.4494 & $5.2015\cdot 10^{-5}$ & $10^{-7}$\\
    % 0.0010 & 1.6921  &  1.2872  &  1.6917 & 0.4185 & 0.1640 \\
    (\ref{update:tau})   & 1.4116   & 1.4106  &  1.4116 & 0.0020 & $10^{-6}$ \\  \hline
    % $0.9\tau^{(k)}$  & 1.4509  &  1.3914  &  1.4509 & 0.0595 & 0.0035 \\
    \end{tabular}%
\end{table}

\end{example}

\begin{example} \rm \label{ep:nonconvex}
Consider the following $2$-player GNEP
%that each player controls $2$ decision variables
\[  %\label{nonconvex}
\begin{array}{cllcl}
    \min\limits_{x_{1}\in\re^2}&(x_{1,1})^3+x_{1,2}x_{2,1}
    +x_{1,1}x_{1,2}+x_{2,2}&\vline&
    \min\limits_{x_{2}\in\re^2}&
    -({x_{2,1}})^4+x_{1,1}({x_{2,2}})^2\\
    \st &({x_{1,1}})^2+({x_{1,2}})^2\le1&\vline&\st&x_{1,1}\le({x_{2,1}})^2+({x_{2,2} })^2\le1.
\end{array}
\]
It can be observed that both objective functions are nonconvex.
Further the feasible set of the second player is not convex neither.
We run Algorithm~\ref{gs} with
$\allx^{(0)} = (0.5,0.5,0.6,0.6)$ and fixed $\tau^{(k)} =0.02$.
For $k=16$, we get
\[
\allx^{(16)} = (-0.9342,\, -0.3568, \, 1.0000, \, -3.7839\cdot 10^{-5}),
\]
which is a GNE. The accuracy $\varepsilon=6.1075\cdot 10^{-8}$.
It costs around $3.68$ seconds.
\end{example}

\begin{example}\rm
  (\cite{facchinei2009generalized,Kesselman2005})
  \label{internet}
Consider the example of a model for Internet switching \cite{facchinei2009generalized,Kesselman2005}.
Assume there are $N$ users, and the maximum capacity of the buffer is $B$.
The $x^i$ denotes the amount of $i$th user's ``packets" in the buffer.
It is clear $x^i\ge0$ for any $i$.
We also suggest the buffer is managed with \textit{``drop-tail" policy},
which means if the buffer is full,
further packets will be lost and resent.
Let $\frac{x_i}{x_1 + \cdots + x_N}$ be the \textit{transmission rate} of user $i$,
and $\frac{x_1 + \cdots + x_N}{B}$ represent the \textit{congestion level} of the buffer,
and $1-\frac{x_1 + \cdots + x_N}{B}$ measure the decrease in the utility of the $i$th user as the congestion level increases.
The $i$th user's optimization problem is
\begin{equation*}
\left\{ \begin{array}{ll}
   \min\limits_{ x_i } &f_i(\allx)=-\frac{x_i}{x_1 + \cdots + x_N}(1-\frac{x_1 + \cdots + x_N}{B}) \\
  %    \left[x_{i,0}-\sum\limits_{j=1}^N\gamma_{j,i}x_{j,i}+2\prod\limits_{k=1}^N
  % (x_{k,0}-\sum\limits_{j=1}^N\gamma_{j,k}x_{j,k})\right],\\
      \st & x_i\ge 0, \, x_1+\cdots x_N\le B.
\end{array} \right.
\end{equation*}
It can be transformed into a polynomial optimization problem by
introducing a new variable $y_i$ for each player.
The GNEP is then equivalent to that
\begin{equation*}
\left\{ \begin{array}{ll}
   \min\limits_{ x_i,y_i } &-x_iy_i(1-\frac{\sum x_i}{B}) \\
  %    \left[x_{i,0}-\sum\limits_{j=1}^N\gamma_{j,i}x_{j,i}+2\prod\limits_{k=1}^N
  % (x_{k,0}-\sum\limits_{j=1}^N\gamma_{j,k}x_{j,k})\right],\\
      \st & x_i\ge0, \, x_1+\cdots x_N\le B\\
          & (x_1+\cdots+x_N)y_i=1.
\end{array} \right.
\end{equation*}
Here, we consider the case that $B=1$ and $N=10$,
and run Algorithm~\ref{gs} with the initial point
\[
(0.4,\underbrace{0.01,0.01,\ldots,0.01}_{9 \,\, \mbox{times} },
  \underbrace{1/0.49,1/0.49,\ldots,1/0.49}_{10 \,\, \mbox{times} }),
\]
and $\tau^{(0)}=0.1$.
The parameters $\tau^{(k)}$ are updated as in \reff{update:tau}.
After $47$ iterations, Algorithm~\ref{gs} returned the point (here we only show the result of $x_1,\dots,x_{10}$)
\[
(0.09, \, 0.09, \,  \ldots, \,  0.09).
\]
with accuracy parameter $\varepsilon=1.6344\cdot 10^{-8}$.
It costs around $61.93$ seconds.
\end{example}

\begin{example} \rm
(\cite[A.~1]{FacKan10})  \label{A1}
Consider a variation of the GNEP in the last example
that we change the constraints of the first player to $0.3\le x_i\le 0.5$.
This GNEP can also be transformed into a GNEPP by introducing a new variable
$y_i$ for each player and it is then equivalent to that
\[
%\label{addconsep1}
\begin{array}{cllcl}
  &\mbox{player }\, i=1 &\vline&&\mbox{player }\, i > 1  \\
  \min\limits_{x_1, y_1 \in \re}& -x_1 y_1 (1-\frac{\sum x_1}{B})
  & \vline & \min\limits_{x_i, y_i \in \re}& -x_i y_i (1-\frac{\sum x_i}{B}) \\
\st & 0.3 \le x_1 \le 0.5 & \vline & \st &  x_1 + \cdots + x_N \le B,\,x_i\ge 0.001\\
  & (x_1 + \cdots + x_N)y_1=1 & \vline & &  (x_1 + \cdots + x_N)y_i=1.
\end{array}
\]
Here, we consider the case that $B=1$ and $N=10$, the same as in \cite{FacKan10}.
We run Algorithm~\ref{gs} with the initial point
\[
(0.3,\underbrace{0.01,0.01,\ldots,0.01}_{9 \,\, \mbox{times} },
\underbrace{1/0.39,1/0.39,\ldots,1/0.39}_{10 \,\, \mbox{times} })
\]
and $\tau^{(0)}=0.1$.
The parameters $\tau^{(k)}$ are updated as in \reff{update:tau}.
After $47$ iterations, Algorithm~\ref{gs} returned the point
(here we only show the result of $x_1,\dots,x_{10}$)
\[
(0.3,  0.06943, 0.06943, \ldots,  0.06943)
\]
with accuracy parameter $\varepsilon=1.1261\cdot 10^{-8}$.
It costs around $60.69$ seconds.
\end{example}

\begin{example} \rm
\label{minusA1}
Consider the GNEP which is the same as in Example~\ref{A1}
except we change the objective function to
\[
f_i(\allx)=\frac{x_i}{x_1 + \cdots + x_N}(1-\frac{x_1 + \cdots + x_N}{B}).
\]
We still consider the case that $B=1$ and $N=10$,
and the same technique to transform each player's subproblem
into polynomial optimization problems.
Start from the initial point
\[
(0.3, \underbrace{0.01,0.01,\ldots,0.01}_{9 \,\, \mbox{times} },
\underbrace{1/0.39,1/0.39,\ldots,1/0.39}_{10 \,\, \mbox{times} })
\]
with $\tau^{(0)}=0.1$ and $\tau^{(k)}$ updated as in \reff{update:tau}.
After $44$ iterations, Algorithm~\ref{gs} returns the GNE
(here we only show the result of $x_1,\dots,x_{10}$)
\[
(0.5000, \, 0.4920, \, 0.0010, \, \ldots, \, 0.0010 )
\]
with the accuracy parameter $\eps=3.7773\cdot 10^{-7}$.
It took about $59$ seconds.
\end{example}

\begin{example} \rm     \label{rdnsimp}
(Random GNEPPs  with joint simplex/ball constraints)
We randomly generate objective polynomials for each player
with the joint simplex/ball constraint
\[
\sum_{i=1}^N\sum_{j=1}^{n_i} x_{i,j}=1, x_{i,j}  \ge  0,
\quad \mbox{ or } \quad
\Vert x_1 \Vert^2 + \cdots + \Vert x_N \Vert^2 \leq 1.
\]
We generate $100$ random instances and count the number of problems
that was solved successfully by Algorithm~\ref{gs}.
The accuracy parameter is set to be $\varepsilon = 10^{-6}$ for
checki{}ng $\allx^{(k)}$ as a GNE, i.e., we regard $(\allx^{(k)})$ 
as a GNE if (\ref{def:acc_para}) was satisfied with $\varepsilon=10^{-6}$.
For each instance, we run Algorithm~\ref{gs} for at most $200$ loops
with $\tau^{(0)}=0.1$, $\tau^{(k+1)}$ updated as in (\ref{update:tau}).
If it does not return a GNE with required accuracy,
we regard that it fails to solve the GNEP.
The performance of Algorithm~\ref{gs} is reported in Table~\ref{jsplx}.
The number $N$ is the number of players,  $n_i$ is the dimension of
the $i$th player's strategy vector,
and $d$ is the degree of objective polynomials.
The time is measured in seconds.

\begin{table}[htb]
\centering
\caption{Computational Results for Example~\ref{rdnsimp}}
\begin{tabular}{|c|l|r|c|r|c|r|} \hline
%  \toprule
\multicolumn{3}{|c|}{}  & \multicolumn{2}{c|}{Joint Simplex} &  \multicolumn{2}{c|}{Joint Ball}  \\  \hline
\multicolumn{1}{|l|}{$N$} & \multicolumn{1}{l|}{$(n_1,\ldots, n_N)$} & \multicolumn{1}{l|}{$d$} & \multicolumn{1}{l|}{Succ. Rate} & \multicolumn{1}{l|}{Ave. Time} &
\multicolumn{1}{l|}{Succ. Rate} & \multicolumn{1}{l|}{Ave. Time}  \\ \hline
3     & (2,2,2)     & 3     &   100\%  &9.97  & 94 \% & 16.71 \\  %\hline
3     & (2,2,2)     & 4     &   92\%  &46.10  & 83 \% & 37.88  \\  %\hline
3     & (3,3,3)     & 2     &   95\%  &11.21  & 97 \% & 9.93  \\  %\hline
3     & (3,3,3)     & 3     &   92\%  &36.21  & 96 \% & 38.44 \\  %\hline
3     & (3,3,3)     & 4     &   84\%  &98.76  & 88 \% & 88.98 \\  %\hline
4     & (3,3,3,3)     & 2     &   94\%  &19.50  & 96 \% & 19.10 \\ %\hline
2     &(4,3)  & 3     &   97\%  &13.53  & 92 \% & 17.55 \\ %\hline
2     &(4,3)  & 4     &   92\%  &52.54  & 94 \% & 55.65 \\ %\hline
3     &(3,2,4)& 2     &   96\%  &9.43   & 97 \% &  9.09 \\ %\hline
3     &(3,2,4)& 3     &   92\%  &44.53  & 98 \% & 26.06 \\ %\hline
4     &(3,2,4,2)& 2     &   93\%  &19.52  & 95 \% & 22.73 \\ %\hline
4     &(3,2,4,2)& 3     &   94\%  &70.76  & 96 \% & 89.46 \\ \hline
\end{tabular}%
\label{jsplx}%
\end{table}%
\end{example}

\subsection{Test problems in \cite{FacKan10}}

We apply Algorithm~\ref{gs} to solve the GNEPs in \cite{FacKan10}
that are GNEPPs or that can be transformed into GNEPPs.
We normalize the objective functions such that the greatest absolute values of the coefficients are equal to one. For example,
the problem~A.17 in \cite{FacKan10} is normalized as follows:
\begin{equation*}
\begin{array}{cllcl}
   \min\limits_{x_1\in \re^2} & \frac{1}{38}((x_{1,1})^2+x_{1,1}x_{1,2}+(x_{1,2})^2)
    &\vline &\min\limits_{x_{2} \in \re^1 } &\frac{1}{25}(x_{1,1}+x_{1,2})x_{2,1}\\
    &\quad+(x_{1,1}+x_{1,2})x_{2,1}-\frac{25x_{1,1}}{38}-x_{1,2}&\vline&&\qquad\quad+\frac{1}{25}(x_{2,1})^2-x_{2,1}\\
  %    \left[x_{i,0}-\sum\limits_{j=1}^N\gamma_{j,i}x_{j,i}+2\prod\limits_{k=1}^N
  % (x_{k,0}-\sum\limits_{j=1}^N\gamma_{j,k}x_{j,k})\right],\\
  \st & x_{1,1}+2x_{1,2}-x_{2,1}\le14,&\vline& \st& x_{1,1}+2x_{1,2}-x_{2,1}\le14,\\
          & 3x_{1,1}+2x_{1,2}+x_{2,1}\le30,&\vline&& 3x_{1,1}+2x_{1,2}+x_{2,1}\le30,\\
          & x_1\ge0, &\vline& & x_1\ge0.\\
\end{array}
\end{equation*}
For the test problem~A.2 and A.14,
we use the same technique as shown in Example~\ref{A1}
to transform these non-polynomial GNEPs into GNEPPs.
For the test problem~A.10a, we run Algorithm~\ref{gs}
with the same initial point as in \cite{FacKan10}
and yield an alternative sequence that is not convergent.
Moreover, we also run Algorithm~\ref{gs} 
with randomly generated feasible initial points for $100$ 
times and no convergent sequence can be obtained.
All the parameters are settled the same as in \cite{FacKan10} for each problem.
The computational results are shown in Table~\ref{tab:FacKan10},
where $e$ denotes the vector of all ones.
All the problems, except problem~A.10a,
were solved successfully by Algorithm~\ref{gs}.
\begin{table}[!htbp]
\centering
\caption{Computational Results for test problems in \cite{FacKan10}}
\label{tab:FacKan10}
\label{test}%
    \begin{tabular}{|r|r|c|c|c|r|c|}  \hline
    {problem}  & initial point& $\tau_0$& $\tau^{(k+1)}$      & iterations    &  {time} & $\varepsilon$        \\ \hline
    A.2        &   $0.05e$  & $0.1$    &(\ref{update:tau})  & 27            &  37.61  & $0.73\cdot10^{-7}$   \\ \hline
    A.3        &   $0.1e$   & $0.1$    &(\ref{update:tau})  & 46            &  13.13  &  $0.10\cdot10^{-5}$  \\ \hline
    A.4        &   $0.1e$   & $0.1$    &(\ref{update:tau})  & 12            &  10.96  & $0.32\cdot10^{-6}$\\ \hline
    A.5        &   $0.1e$   & $0.1$    &(\ref{update:tau})  & 25            &  10.21  & $0.16\cdot10^{-7}$\\ \hline
    A.6        &   $e$      & $0.1$ 	 &(\ref{update:tau})  & 38            &  11.38  & $0.41\cdot10^{-6}$\\ \hline
    A.7        &   $e$      & $0.1$ 	 &(\ref{update:tau})  & 17            &  10.74  & $0.21\cdot10^{-7}$\\ \hline
    A.8        &   $0.5e$   & $0.1$    &(\ref{update:tau})  & 54           &  14.84  & $0.52\cdot10^{-6}$\\ \hline
    A.10a        &   see \cite{FacKan10}  & $0.1$    &(\ref{update:tau})  & 200         &  \multicolumn{2}{|c|}{not convergent}\\ \hline
    A.11       &   $0.5e$   & $0.1$  	 &(\ref{update:tau})  & 37            &   4.86  & $0.11\cdot10^{-6}$\\ \hline
    A.12       &   $e$      & $0.1$	 &(\ref{update:tau})  & 65            &   7.22  & $0.17\cdot10^{-7}$\\ \hline
    A.13       &   $e$      & $0.1$	 &(\ref{update:tau})  & 12           &   2.01  & $0.71\cdot10^{-8}$\\ \hline
    A.14       &   $0.1e$  &  $0.1$  	 &(\ref{update:tau})  & 42            &   50.50  & $0.56\cdot10^{-8}$\\ \hline
    A.15       &   $e$    &  $0.0001$	 &(\ref{update:tau})  & 200           &  45.51  & $0.13\cdot10^{-5}$\\ \hline
    A.17       &   $e$    & 	 $0.001$    &(\ref{update:tau})  & 200            &   25.82  & $0.19\cdot10^{-7}$\\ \hline
    A.18       &   $e$  & $0.5$  	 &(\ref{update:tau})  & 200           &   59.11  & $0.29\cdot10^{-5}$\\ \hline
    \end{tabular}%
\end{table}

\section{Conclusions}
This paper discusses how to use the Gauss-Seidel method for solving the
generalized Nash equilibrium problems of polynomials.
The polynomial optimization in each loop of the Gauss-Seidel method
is solved by the Lasserre type Moment-SOS relaxations.
The convergence results are presented for general GNEPs
and for the special case of GPGs.
In particular, we give a certificate for checking GPGs.
Numerical experiments show that the Gauss-Seidel method
is efficient for solving many GNEPPs,
even if the players' optimization problems are nonconvex.

\end{document}